\documentclass[onefignum,onetabnum]{siamart171218}

\usepackage{amsfonts}
\usepackage{amssymb}
\usepackage{graphicx}
\usepackage{epstopdf}
\usepackage{algorithmic}
\ifpdf
  \DeclareGraphicsExtensions{.eps,.pdf,.png,.jpg}
\else
  \DeclareGraphicsExtensions{.eps}
\fi
\usepackage{bm}
\usepackage{tikz-cd}
\usepackage{subfig}
\usepackage{mathtools}
\usepackage{stmaryrd}
\newcommand\commentone[1]{\textcolor{black}{#1}}
\newcommand\commenttwo[1]{\textcolor{black}{#1}}
\newsiamremark{remark}{Remark}
\newsiamremark{assumption}{Assumption}

\newsiamremark{hypothesis}{Hypothesis}
\crefname{hypothesis}{Hypothesis}{Hypotheses}
\newsiamthm{claim}{Claim}

\def\XXint#1#2#3{{\setbox0=\hbox{$#1{#2#3}{\int}$ }
\vcenter{\hbox{$#2#3$ }}\kern-.6\wd0}}

\headers{Monotone Fractional Laplacian Scheme: Pointwise Error
Estimate}{R. Han and S. Wu}

\title{A monotone discretization for integral fractional Laplacian on
bounded Lipschitz domains: pointwise error estimates under H\"{o}lder
regularity
\thanks{
The work of Shuonan Wu is supported in part by the National Natural
Science Foundation of China grant No. 11901016 and the startup grant
from Peking University. }
}

\author{
Rubing Han\thanks{School of Mathematical Sciences,
  Peking University, Beijing 100871, China} \and 
Shuonan Wu\thanks{School of Mathematical Sciences, Peking University,
Beijing 100871, China.}
}

\usepackage{amsopn}

\ifpdf
\hypersetup{
  pdftitle={fractional Laplacian: Monotone FE + Pointwise Estimate},
  pdfauthor={R.Han and S. Wu}
}
\fi

\begin{document}

\maketitle

\begin{abstract}
We propose a monotone discretization for the integral
fractional Laplace equation on bounded Lipschitz domains with
the homogeneous Dirichlet boundary condition. The method is inspired by a
quadrature-based finite difference method of Huang and Oberman, but
is defined on unstructured grids in arbitrary dimensions with a more
flexible domain for approximating singular integral. The scale of
the singular integral domain not only depends on the local grid size,
but also on the distance to the boundary, since the H\"{o}lder
coefficient of the solution deteriorates as it approaches the
boundary.  By using a discrete barrier function that also reflects the
distance to the boundary, we show optimal pointwise convergence rates
in terms of the H\"{o}lder regularity of the data on both quasi-uniform and graded grids. Several numerical
examples are provided to illustrate the sharpness of the theoretical
results.
\end{abstract}



\section{Introduction} \label{sc:intro}
It is known that the fractional Laplacian can be defined in many ways \cite{kwasnicki2017ten, lischke2020fractional}, which are not equivalent on the bounded Lipschitz domain $\Omega \subset \mathbb{R}^n$. In this work, we focus on the integral fractional Laplace equation  of form \cite{DINEZZA2012521}
\begin{equation} \label{eq:fL}
\left\{
  \begin{aligned}
    \mathcal{L} u:= (-\Delta)^s u &= f \quad \text{in }\Omega, \\
    u &= 0 \quad \text{in }\Omega^c := \mathbb{R}^n \setminus \Omega,
  \end{aligned}
\right.
\end{equation}
where $(-\Delta)^s$ is the integral fractional Laplacian of order $s \in (0,1)$, defined by
\begin{equation} \label{eq:L}
  (-\Delta)^s u(x) := C_{n,s}\, \mathrm{P.V.} \int_{\mathbb{R}^n} 
  \frac{u(x) - u(y)}{|x-y|^{n+2s}}\,\mathrm{d}y.
\end{equation}
The normalization constant is given by 
$C_{n,s} := \frac{2^{2s} s \Gamma(s+\frac{n}{2})}{\pi^{n/2}
\Gamma(1-s)}$. In probability, the fractional Laplacian is the
infinitesimal generator of a symmetric $2s$-stable L\'{e}vy process
\cite{applebaum2009levy}. Meanwhile, the fractional Laplacian has been
used in place of the integer-order Laplacian in many applications,
including financial asset prices \cite{tankov2003financial},
fractional conservation law \cite{droniou2010numerical}, geophysical
fluid dynamics \cite{cordoba2004maximum}.

Besides the non-locality and singularity in the kernel, another main
feature of integral fractional Laplacian is that the solutions to
\eqref{eq:fL} exhibit an algebraic boundary singularity regardless of
the domain regularity. As a well-known fact, even if domain is smooth,
the unique solution to \eqref{eq:fL}  has only the optimal
$C^s$-H\"{o}lder regularity in $\bar{\Omega}$ provided that $f\in
L^\infty(\Omega)$, and develops a singularity of the form
$\mathrm{dist}(x,\partial\Omega)^s$ near $\partial \Omega$. These
properties are not limited to \eqref{eq:fL}, but widely appear in a
class of nonlocal elliptic equations, see a survey in
\cite{ros2016nonlocal}. Moreover, when the data has better regularity,
the higher-order interior H\"{o}lder regularity of $u$ requires the use of
weighted H\"{o}lder norms, in which the distance to the boundary is
also involved \cite{ros2014dirichlet}. 

Numerical studies of \eqref{eq:fL} and related fractional-order
problems have experienced some rapid developments in recent years.
Based on the variational formula, a finite element (FE) discretization
using piecewise linear continuous function space was developed and
analyzed in \cite{acosta2017fractional}. Significantly, the error
estimates in energy norm on quasi-uniform and graded grids hinge on
the standard and weighted Sobolev regularity, respectively.  Other
works related to the conforming FE discretizations include error
estimate in $H^1$ norm in the case $s > \frac12$
\cite{borthagaray2019convergence}, local energy estimate
\cite{borthagaray2021local}, multilevel solver
\cite{borthagaray2021robust}, and extensions to non-homogeneous
Dirichlet problem \cite{acosta2019finite} and eigenvalue problem
\cite{borthagaray2018finite}. A nonconforming FE discretization,
based on the Dunford-Taylor representation, was proposed and analyzed
in \cite{bonito2019numerical}. Discretizations of the spectral
fractional Laplacian can be found in
\cite{nochetto2015pde,cusimano2018discretizations}. We refer to
\cite{bonito2018numerical, lischke2020fractional} for the survey of
existing numerical methods for fractional Laplacian.

The maximum principle, even in the case of weak solutions to
\eqref{eq:fL}, is valid \cite{ros2016nonlocal}. In view of numerical
stability, it is desirable that the resulting discrete system also
satisfies a similar maximum principle at discrete level --- known as
the {\it monotonicity}. As in the case of the Laplacian $-\Delta$, the
monotonicity of FE discretizations essentially relies on some grid
conditions \cite{xu1999monotone}, for instance, the Delaunay
triangulation in 2D. Such grid conditions, however, are extremely
difficult to obtain and verify for the FE discretizations of integral
fractional Laplacian mainly due to its nonlocal nature. 

On the other hand, the finite difference (FD) methods have natural
advantages in constructing monotone schemes \commentone{\cite{motzkin1952approximation,kocan1995approximation}}. For instance, a criterion
for easy verification of monotonicity was proposed in
\cite{oberman2006convergent} in the FD setting. In
\cite{huang2014numerical}, Huang and Oberman first proposed a
quadrature-based FD method for solving 1D integral fractional Laplacian. The
monotonicity was proved when using piecewise linear function space,
which yields convergence. In this case, the accuracy of the scheme in
$L^\infty$ norm was shown to be $\mathcal{O}(h^{2-2s})$ under the
assumption that the solution of \eqref{eq:fL} belongs to $C^4$, which
cannot be guaranteed in general. In fact, a simple numerical example
in Table \ref{tb:huang-oberman} shows that the convergence rate in
$L^\infty$ behaves like $\mathcal{O}(h^{\min\{s,2-2s\}})$, where $h$
represents the uniform grid size. The bottleneck of $\mathcal{O}(h^s)$
convergence rate at most also exists in other FD methods for
\eqref{eq:fL} \cite{duo2018novel, duo2019accurate}.

\begin{table}[!htbp]
\centering
	\begin{tabular}{|c||c|c|c|c|c|c|c|c|c|c|}
		\hline
		Value of $s$& 0.10 & 0.20 & 0.30 & 0.40 & 0.50 & 0.60 & 0.67&0.70 & 0.80 & 0.90\\
		\hline
		Order (in $h$) & 0.10 & 0.20 & 0.30 & 0.40 & 0.50 & 0.60 & 0.67& 0.64 & 0.42 &0.21\\
		\hline
	\end{tabular}
	\caption{$\Omega = (-1,1)$, $f =1$. The pointwise convergence order of Huang-Oberman scheme \cite{huang2014numerical} is $\approx  
	\min\{s, 2-2s\}$.}  \label{tb:huang-oberman}
\end{table}

\vspace{-6mm}

The primal goal of this paper is to develop a monotone scheme with
optimal convergence rates for \eqref{eq:fL} on bounded Lipschitz
domains. To this end, we first triangulate the domain $\Omega$ by a
shape regular grid, and construct the associated piecewise linear
continuous function space where the approximate solution lies.
Inspired by \cite{huang2014numerical}, the integral fractional
Laplacian \eqref{eq:L} is divided into the singular part and the tail
part for every interior grid node $x_i$. Under some symmetry
conditions for the singular integral domain $\Omega_i$, the singular
integral can be approximated by a local differential operator to which
the standard monotone FD discretization can be applied. This, together
with a monotone discretization of the tail part and a discrete barrier
function, leads to the discrete comparison principle. 

To achieve the optimal convergence rates, several key ideas are
introduced and utilized in this paper. First, we show the relationship
between high-order H\"{o}lder constant of solution and distance to
boundary, which indicates that the scale of $\Omega_i$, denoted by
$H_i$, should also depend on $\mathrm{dist}(x_i, \partial \Omega)$.
Second, upon the ratio between $\mathrm{dist}(x_i, \partial \Omega)$
and local grid size, we introduce the {\it $\delta$-interior nodes} in
\eqref{eq:interior-node} where the consistency can be proved. This,
together with a discrete barrier function and the discrete comparison
principle, leads to the optimal error estimates. The
analysis only requires the local quasi-uniformity hence can be applied
to both quasi-uniform and graded grids. Moreover, the analysis only
requires the minimal H\"{o}lder regularity of the data, and
interestingly, the H\"{o}lder regularity indeed influences the optimal
scale of $\Omega_i$, see Theorem \ref{tm:err-h} and Theorem
\ref{tm:err-N}.

It is interesting to note that, there is a strong connection between our
work and the monotone two-scale methods for solving Monge-Amp\`ere equations \cite{nochetto2019two}. \commentone{We note that the two-scale methods are also important for problems in which preserving monotonicity and comparison is of relevance \cite{motzkin1952approximation,kocan1995approximation}}. In our work, the scale of $\Omega_i$
can be viewed as the second scale in addition to grid size, which
depends on the distance to boundary due to the PDE theory of
\eqref{eq:fL}.  Moreover, it is known that monotonicity is one of the
essential ingredients in devising convergent schemes for fully
nonlinear PDEs \cite{barles1991convergence}. Our work, in this sense,
is expected to have broad prospects in solving nonlinear problems
involving integral fractional Laplacian.

The rest of the paper is organized as follows. In Section
\ref{sc:regularity}, we review the regularity results for the integral
fractional Laplace equation \eqref{eq:fL}. In Section
\ref{sc:discretization}, we  introduce a monotone discretization and
prove the discrete comparison principle. The consistency error is
divided into singular and tail parts, which are estimated in Section
\ref{sc:Lhs} and Section \ref{sc:Lht} separately. In Section
\ref{sc:error}, combining the discrete comparison principle and the
consistency error, we establish the pointwise error estimates for the
numerical scheme. Finally, in Section \ref{sc:numerical}, several
numerical examples are exhibited to illustrate the theoretical
results.

\section{Preliminary results} \label{sc:regularity}
In this section, we present some preliminary results in the analysis
setting.  For $\beta > 0$, we denote by $|\cdot|_{C^\beta(U)}$ the
$C^\beta(U)$ seminorm. More precisely, we will write $\beta = k +
\beta'$ with $k$ integer and $\beta' \in (0,1]$, then 
$$
  \begin{aligned}
    |w|_{C^\beta(U)} = |w|_{C^{k,\beta'}(U)} &:= \sup_{x,y \in U, x\neq y}
  \frac{|D^kw(x) - D^k w(y)|}{|x - y|^{\beta'}}, \\
    \|w\|_{C^\beta(U)} & := \sum_{\ell = 0}^k \left(\sup_{x\in U}
    |D^\ell w(x)| \right) + |w|_{C^\beta(U)}.
  \end{aligned}
$$

Next, we summarize the H\"{o}lder regularity results for \eqref{eq:fL} given in
\cite{ros2014dirichlet}. To begin with, we state the definition of
Lipschitz domain (cf. \cite[Definition
1.2.1.1]{grisvard2011elliptic}).

\begin{definition}[Lipschitz domain] \label{df:Lipschitz}
Let $\Omega$ be an open subset of $\mathbb{R}^n$. We say that $\Omega$
is a Lipschitz domain if for every $x \in \partial \Omega$ there
exists a neighborhood $V$ of $x$ in $\mathbb{R}^n$ and new
orthogonal coordinates $\{z_1,\cdots,z_n \}$ such that
\begin{itemize}
\item[(a)] $V$ is an hypercube in the new coordinates:
\[
V = \{(z_1,\cdots, z_n) : -a_j < z_j< a_j,~ 1\leq j\leq n  \};
\]
\item[(b)] there exists a Lipschitz function $\varphi$, defined in 
\[
V':= \{(z_1,\cdots, z_{n-1}) : -a_j< z_j < a_j,\; 1\leq j \leq n-1 \}
\]
and such that
\[
\begin{aligned}
  &|\varphi(z')| \leq a_n /2 \text{ for every } z' = (z_1,\cdots,z_{n-1})
  \in V', \\
  &\Omega \cap V = \{z = (z',z_n) \in V :  z_n < \varphi(z') \},\\
  &\partial \Omega\cap V = \{z = (z',z_n) \in V : z_n =
  \varphi(z') \}.
\end{aligned}
\]
\end{itemize}

\end{definition}

\begin{proposition}[Proposition 1.1 in \cite{ros2014dirichlet}]
\label{pp:global-holder} 
Let $\Omega$ be a bounded Lipschitz domain
satisfying exterior ball condition, $f \in L^\infty(\Omega)$, and $u$
be a solution of \eqref{eq:fL}. Then $u \in C^s(\mathbb{R}^n)$ and 
$$ 
  \|u\|_{C^s(\mathbb{R}^n)} \leq C \|f\|_{L^\infty(\Omega)},
$$ 
where $C$ is a constant depending only on $\Omega$ and $s$.
\end{proposition}

Moreover, the higher-order H\"{o}lder regularity inside $\Omega$ can
be obtained when $f$ is H\"{o}lder continuous, which hinges on the
weighted norm up to the boundary.  Let $\delta(x)$ be the distance to
the boundary of $x$, and $\delta(x,y) := \min\{\delta(x), \delta(y)\}$.
Following \cite{ros2014dirichlet}, let $\beta > 0$ and $\sigma \geq -\beta$, we define
the seminorm 
\begin{equation} \label{eq:w-holder-semi}
  |w|_{\beta;U}^{(\sigma)} := \sup_{x,y \in U} \left(
  \delta(x,y)^{\beta + \sigma} \frac{|D^kw(x) - D^kw(y)|}{|x -
  y|^{\beta'}} \right)
  \quad \forall w \in C^\beta(U) := C^{k,\beta'}(U).
\end{equation}
For $\sigma > -1$, the associated norm
$\|\cdot\|_{\beta;U}^{(\sigma)}$ is defined as follows: 
\begin{subequations}\label{eq:w-holder}
\begin{enumerate}
\item For $\sigma \geq 0$,
  \begin{equation} \label{eq:w-holder1} 
  \|w\|_{\beta;U}^{(\sigma)} := \sum_{\ell=0}^k \sup_{x\in U} 
  \left( \delta(x)^{\ell+\sigma} |D^\ell w(x)|\right) +
  |w|_{\beta;U}^{(\sigma)};
  \end{equation}
\item For $0 > \sigma > -1$, 
\begin{equation} \label{eq:w-holder2}
  \|w\|_{\beta;U}^{(\sigma)} := \|w\|_{C^{-\sigma}(U)} +
  \sum_{\ell=1}^k  \sup_{x\in U} 
  \left( \delta(x)^{\ell+\sigma} |D^\ell w(x)|\right) +
  |w|_{\beta;U}^{(\sigma)}. 
\end{equation}
\end{enumerate}
\end{subequations}

The following result is the starting point of this work.
\begin{proposition}[Proposition 1.4 in \cite{ros2014dirichlet}]
  \label{pp:delta-holder} 
  Let $\Omega$ be a bounded domain, and $\beta > 0$ be such that neither 
  $\beta$ nor $\beta+2s$ is an integer. Let $f \in C^\beta(\Omega)$ be
  such that $\|f\|_{\beta;\Omega}^{(s)} < \infty$, and $u \in
  C^s(\mathbb{R}^n)$ be a solution of \eqref{eq:fL}. Then, $u \in
  C^{\beta + 2s}(\Omega)$ and 
  \begin{equation} \label{eq:delta-holder} 
  \|u\|_{\beta+2s;\Omega}^{(-s)} \leq C(\|u\|_{C^s(\mathbb{R}^n)} +
  \|f\|_{\beta;\Omega}^{(s)}),
  \end{equation}
where $C$ is a constant depending only on $\Omega$, $s$ and $\beta$.
\end{proposition}

\begin{remark}[blow-up behavior] \label{rk:blow-up-speed}
We further note that the blow-up behavior of the constant $C$ in \eqref{eq:delta-holder} is $|\beta - m|^{-1}$ or $|\beta
+ 2s -m|^{-1}$ as $\beta$ or $\beta +
2s$ approaching to some integer $m$, which can be found in the proof of \cite[Proposition
2.5 \& Proposition 2.7]{silvestre2007regularity}.
\end{remark}

As a corollary, the dependence of the $\delta(x)$ on the H\"{o}lder
norm of $u$ is given as follows. For the sake of expository simplicity, we adopt $\beta$ as the H\"{o}lder index of $u$. We also assume $\Omega$ to  be a bounded Lipschitz domain with exterior ball condition in what follows.

\begin{corollary}[$\delta$-dependence in H\"{o}lder norm] 
\label{co:rho-holder}
Let $\Omega$ be a bounded Lipschitz domain with exterior ball
condition, and $\beta > 2s$ be such that neither $\beta - 2s$ nor $\beta$ is
an integer. Let $f \in L^\infty(\Omega)$ be such that
$\|f\|_{\beta-2s;\Omega}^{(s)} < \infty$. Then, $u \in
C^\beta(\Omega)$ and
\begin{equation} \label{eq:rho-holder}
  \|u\|_{C^\beta(\{x\in \Omega:~\delta(x) \geq \rho\})} \leq C\rho^{s
  - \beta} \quad \forall \rho \in(0, \mathrm{diam}(\Omega)], 
\end{equation}
where $C$ is a constant depending only on $\Omega$, $s$, $\beta$,
$\|f\|_{L^\infty(\Omega)}$ and $\|f\|_{\beta-2s;\Omega}^{(s)}$.
\end{corollary}
\begin{proof}
We denote $U_\rho := \{x\in \Omega:~\delta(x) \geq \rho\}$.  Using
Proposition \ref{pp:delta-holder} and the definition of weighted
seminorm \eqref{eq:w-holder-semi}, we have 
$$ 
  \begin{aligned}
    |u|_{C^\beta(U_\rho)} &= \rho^{s-\beta} \sup_{x,y\in U_\rho}
    \rho^{\beta - s} \frac{|D^k u(x) - D^ku(y)|}{|x-y|^{\beta'}} \\
    & \leq \rho^{s-\beta} \sup_{x,y\in U_\rho} \delta(x,y)^{\beta - s}
    \frac{|D^k u(x) - D^ku(y)|}{|x-y|^{\beta'}} \quad(\text{since }\beta > s)\\
    & \leq \rho^{s-\beta} |u|_{\beta;U_\rho}^{(-s)} \leq
    C(\Omega,s,\beta) \rho^{s-\beta}(\|u\|_{C^s(\mathbb{R}^n)} +
    \|f\|_{\beta-2s;\Omega}^{(s)}) \\ 
    & \leq C(\Omega,s,\beta,\|f\|_{L^\infty(\Omega)},
    \|f\|_{\beta-2s;\Omega}^{(s)}) \rho^{s - \beta}. 
  \end{aligned}
$$ 
Here, we use Proposition \ref{pp:global-holder} in the last step.
Similarly, for any $1 \leq \ell \leq k<\beta$, 
$$ 
  \sup_{x \in U_\rho}|D^\ell u(x)| \leq \rho^{s -
  \ell} \|u\|_{\beta;U_\rho}^{(-s)} \leq
  C(\Omega,s,\beta,\|f\|_{L^\infty(\Omega)},
  \|f\|_{\beta-2s;\Omega}^{(s)})\rho^{s - \beta}, 
$$ 
where we use $\rho^{\beta - \ell} \leq \mathrm{diam}(\Omega)^{\beta
- \ell}$ in the last step. This completes the proof.
\end{proof}

At this point, we recall a useful estimate corresponding to the
integrability of kernel outside the domain (cf. \cite[Eq.
(1,3,2,12)]{grisvard2011elliptic}): There exist two constants $0 <
C_1 \leq C_2$ such that 
\begin{equation} \label{eq:kernel-int}
C_1(\Omega,s) \delta(x)^{-2s} \leq \int_{\Omega^c} \frac{1}{|x -
y|^{n+2s}} \,\mathrm{d}y \leq C_2(\Omega,s)\delta(x)^{-2s}.
\end{equation}

In the last of this section, we present an elementary result of the center
difference in H\"{o}lder seminorm along the direction $\theta \in
S^{n-1}$, where $S^{n-1}$ denotes the unit $(n-1)$-sphere.
\begin{lemma}[FD in H\"{o}lder seminorm] \label{lm:fd-holder}
Let $\rho > 0$ and $u \in C^{\beta}(B_\rho(x))$.
\begin{enumerate}
\item If $\beta \leq 2$, then there exists $C>0$ for any $\theta \in
  S^{n-1}$,
$$ 
|2u(x) - u(x+\rho \theta) - u(x-\rho\theta)| \leq C\rho^\beta
|u|_{C^\beta(B_\rho(x))}.
$$ 
\item If $2< \beta \leq 4$, then there exists $C>0$ for any
$\theta \in S^{n-1}$,
$$ 
\left|2u(x) - u(x+\rho \theta) - u(x-\rho\theta) + \rho^2
\frac{\partial^2 u}{\partial \theta^2}(x)\right| \leq C \rho^\beta
|u|_{C^\beta(B_\rho(x))}.
$$ 
\end{enumerate} 
\end{lemma}
\begin{proof}
Recall that $\beta = k + \beta'$ for integer $k \leq 3$ and $\beta'
\in (0,1]$, we have by Taylor's expansion with Lagrange remainder term that,
for some $\xi \in B_\rho(x)$,
$$
\begin{aligned}
u(x + \rho \theta) & = u(x) + \sum_{j=1}^{k-1}
  \frac{1}{j!}\frac{\partial^j u}{\partial \theta^j}(x) \rho^j +
  \frac{1}{k!}\frac{\partial^k u}{\partial \theta^k}(\xi) \rho^k \\
&= u(x) + \sum_{j=1}^{k} \frac{1}{j!}\frac{\partial^j u}{\partial
  \theta^j}(x) \rho^j +  \frac{1}{k!}\left(\frac{\partial^k
  u}{\partial \theta^k}(\xi) - \frac{\partial^k u}{\partial
  \theta^k}(x)\right) \rho^k. 
\end{aligned}
$$ 
The last term can be controlled by using the definition of H\"{o}lder
  seminorm, which completes the proof.
\end{proof}

\section{Monotone discretization} \label{sc:discretization}
\commentone{For ease of discretization and the sake of simplicity of the exposition, we shall henceforth consider $\Omega$ to be a bounded open polytope.} We assume $\Omega$ admits an admissible triangulation $\mathcal{T}_h$,
i.e., \commentone{$\cup_{T\in \mathcal{T}_h}\bar{T} =\overline{\Omega}$}.  Let $\mathcal{N}_h$ denote
the node set of $\mathcal{T}_h$, $\mathcal{N}_h^b := \{x_i \in
\mathcal{N}_h: x_i \in \partial \Omega\}$ be the collection of boundary nodes, and
$\mathcal{N}_h^0 := \mathcal{N}_h \setminus \mathcal{N}_h^b$. We
require that the family of triangulation under consideration satisfies
shape regularity and local quasi-uniformity. \commentone{Note that the latter can be deduced by the former when $n \geq 2$ \cite{brenner2008mathematical}}.  More precisely, let $h_T$ and
$\rho_T$ respectively be the diameter of $T$ and diameter of the
largest ball contained in $T$, the conditions of triangulation read
\begin{subequations} \label{eq:triangulation}
\begin{align}
\exists \lambda_1 > 0 ~&~\mbox{s.t. } h_T \leq \lambda_1 \rho_T ~\forall T
  \in \mathcal{T}_h, \label{eq:shape-reg1} \\
\exists \lambda_2 > 0 ~&~ \mbox{s.t. } h_T \leq \lambda_2 h_{T'} ~\forall
  T, T'\in\mathcal{T}_h \mbox{ with } \bar{T} \cap \bar{T}' \neq
  \varnothing. \label{eq:shape-reg2} 
\end{align}
\end{subequations}
For each $x_i \in \mathcal{N}_h^0$, we define $\omega_i =
\cup_{\bar{T}\ni x_i} \bar{T}$, and denote by $h_i$ the radius of
inscribed sphere centered at $x_i$ for $\omega_i$. Clearly, for any
$\bar{T} \ni x_i$, 
\begin{equation} \label{eq:shape-reg3} 
\underbrace{2\lambda_1\lambda_2}_{:=\lambda} h_i \geq
  \lambda_1\lambda_2 \rho_T \geq \lambda_2 h_T \geq
\max_{\bar{T'}\ni x_i}h_{T'}.
\end{equation} 

Let $\mathbb{V}_h$ be the space of continuous piecewise linear
functions over $\mathcal{T}_h$ that vanish in $\Omega^c$, that is, 
\begin{equation} \label{eq:FEM-space}
\mathbb{V}_h := \{ v:\; v|_T \in \mathcal{P}_1(T) ~\forall T \in
\mathcal{T}_h,\; v|_{\Omega^c} = 0\}.
\end{equation} 

Following \cite{huang2014numerical}, we divide the integral of fractional Laplacian into two
parts: The singular part and the tail part.

\subsection{Discretization of singular integral}
For each $x_i \in \mathcal{N}_h^0$, we take a proper scale $H_i$ such that $B_{H_i}(x_i) \subset \Omega$. The standard centered
difference of the $\Delta u(x_i)$ with spacing $H_i$ is denoted by 
$$
\frac{\Delta_{\rm FD} u(x_i;H_i)}{H_i^2} := \sum_{j=1}^n \frac{u(x_i +
H_i e_j) - 2u(x_i) + u(x_i - H_i e_j)}{H_i^2},
$$
where $e_j$ is the unit vector of the $j$-th coordinate.

For the discretization of singular integral, we consider a star-shaped
domain $\Omega_i$ centered at $x_i$, i.e., by using the
multi-dimensional polar coordinate 
\begin{equation} \label{eq:singular-domain}
\Omega_i := \{x_i + \rho \theta: \theta \in S^{n-1}, \rho \in [0,
  \rho_i(\theta)), \rho_i(\theta) > 0 \}.
\end{equation}
The domain $\Omega_i$ is assumed to satisfy the following conditions:
\begin{enumerate}
\item Interior of $\Omega$: $\Omega_i \subset \Omega$, $\forall x_i \in
  \mathcal{N}_h^0$.
\item Symmetry: 
\begin{equation} \label{eq:theta-symmetry}
\begin{aligned} 
& x_i + (z_1, \cdots, z_n) \in \Omega_i \\
\Longleftrightarrow ~& x_i + (\pm z_{\sigma(1)}, \cdots, \pm
  z_{\sigma(n)}) \in \Omega_i \mbox{ for any permutation } \sigma.
\end{aligned}
\end{equation}
\item Quasi-uniformity: there exist positive constants
  $\underline{c}_S, \bar{c}_S$, such that 
\begin{equation} \label{eq:theta-uniformity}
\underline{c}_S H_i \leq \rho_i(\theta) \leq \bar{c}_S H_i \quad
  \forall  x_i \in \mathcal{N}_h^0, \theta \in S^{n-1}.
\end{equation}
\end{enumerate}

Using the symmetry of $\Omega_i$ and polar coordinate transformation
$y \mapsto x_i + z = x_i + \rho \theta$ which satisfies $\mathrm{d}z =
\rho^{n-1} \mathrm{d} \rho \mathrm{d} S_\theta$, we have
\begin{equation} \label{eq:reflective-sym}  
\begin{aligned}
\int_{\Omega_i} \frac{v(x_i) - v(y)}{|x_i - y|^{n+2s}} \,\mathrm{d}y 
&=  \frac{1}{2}\int_{\Omega_i - x_i} \frac{2v(x_i) - v(x_i+z) -
  v(x_i-z)}{|z|^{n+2s}} \,\mathrm{d}z  \\
& =  \frac12  \commentone{\int_{S^{n-1}}\int_0^{\rho_i(\theta)}}
  \frac{2v(x_i) - v(x_i + \rho \theta) - v(x_i -
  \rho\theta)}{\rho^{1+2s}} \, \commentone{\mathrm{d}\rho}\mathrm{d}S_\theta.
\end{aligned}
\end{equation}

Suppose the function $v \in C^2(\Omega_i)$, the above integral can be
approximated by the following second-order differential operator  
$$ 
\begin{aligned}
& \frac12 \commentone{\int_{S^{n-1}}\int_0^{\rho_i(\theta)}} \, 
  \frac{2v(x_i) - v(x_i + \rho \theta) - v(x_i -
  \rho\theta)}{\rho^{1+2s}} \,\commentone{\mathrm{d}\rho}\mathrm{d}S_\theta\\
\approx ~& -\frac12 D^2 u(x_i) :   \commentone{\int_{S^{n-1}} \int_0^{\rho_i(\theta)} \rho^{1-2s}\mathrm{d}\rho\;
    \theta \otimes \theta \,\mathrm{d}
  S_\theta} \\
= ~& -\frac{H_i^{2-2s}}{4(1-s)} D^2 u(x_i) : \int_{S^{n-1}} \left(
  \frac{\rho_i(\theta)}{H_i} \right)^{2-2s} \theta \otimes \theta
  \,\mathrm{d} S_\theta.
\end{aligned}
$$ 
Thanks to the symmetry of $\Omega_i$, the last integral can be
simplified by the lemma below. 
\begin{lemma}[symmetric integral] Under the conditions
  \eqref{eq:theta-symmetry} -- \eqref{eq:theta-uniformity}, 
\begin{equation} \label{eq:integral-symmetry}
\frac{1}{4(1-s)} \int_{S^{n-1}} \left( \frac{\rho_i(\theta)}{H_i}
  \right)^{2-2s} \theta \otimes \theta \,\mathrm{d}S_\theta 
= \underbrace{\frac{1}{4n(1-s)} \int_{S^{n-1}}  \left(
  \frac{\rho_i(\theta)}{H_i} \right)^{2-2s} \,\mathrm{d}S_\theta}_{:=
  \kappa_{n,s,i}}  I_n,
\end{equation}
where $I_n$ is the identity matrix. Moreover, $\kappa_{n,s,i}$ is a
uniformly bounded constant.
\end{lemma}
\begin{proof}
It is readily seen that $\rho_i(\pm\theta_{\sigma(1)}, \cdots,
  \pm\theta_{\sigma(n)})$ are the same.  For any fixed permutation
  $\sigma$, let $\theta_{j_1,\cdots, j_n}^\sigma := ((-1)^{j_1}
  \theta_{\sigma(1)}, \cdots, (-1)^{j_n} \theta_{\sigma(n)})$, then 
$$ 
\sum_{j_1,\cdots, j_n = 0}^{1} \theta_{j_1,\cdots,j_n}^\sigma \otimes
  \theta_{j_1,\cdots, j_n}^\sigma = 2^n
  \mathrm{diag}(\theta^{\commentone{2}}_{\sigma(1)}, \cdots, \theta^{\commentone{2}}_{\sigma(n)}),
$$ 
which together with $\theta_{j_1,\cdots, j_n}^\sigma \in S^{n-1}$
  yields 
$$ 
\sum_{\sigma} \sum_{j_1,\cdots, j_n = 0}^{1}
  \theta_{j_1,\cdots,j_n}^\sigma \otimes \theta_{j_1,\cdots,
  j_n}^\sigma = \frac{2^n n!}{n} I_n.
$$ 
Hence, the contribution of $\theta \otimes \theta$ in the integral is
identically to $\frac{1}{n} I_n$, which leads to
\eqref{eq:integral-symmetry}.  The quasi-uniformity
  \eqref{eq:theta-uniformity} implies that $\kappa_{n,s,i} =
  \mathcal{O}(1)$.
\end{proof}

Combining \eqref{eq:reflective-sym} -- \eqref{eq:integral-symmetry},
the discretization of the singular integral is defined by 
\begin{equation} \label{eq:Ls}
\begin{aligned}
\mathcal{L}_{h}^S[u](x_i) & := -\kappa_{n,s,i} \frac{\Delta_{\rm FD}
  u(x_i; H_i)}{H_i^{2s}} \\
&= -\kappa_{n,s,i}\sum_{j=1}^n \frac{u(x_i + H_i e_j) - 2u(x_i) +
  u(x_i - H_i e_j)}{H_i^{2s}}.
\end{aligned}
\end{equation}

\begin{remark}[dependence of $H_i$] So far the discretization of
  singular integral depends \commentone{solely} on $\Omega_i$ with scale $H_i$. We
  will show below that the $H_i$ should depend on local grid size and the distance
  of $x_i$ to the boundary. Therefore, for the sake of consistency, we
  adopt $\mathcal{L}_h^S$ to indicate the essential dependence
  on $h$.
\end{remark}

\begin{remark}[examples of $\Omega_i$] \label{rk:omega-i}
The simplest example that
  satisfies the condition \eqref{eq:theta-symmetry} --
  \eqref{eq:theta-uniformity} is $\Omega_i = B_{H_i}(x_i)$, which
  gives $\kappa_{n,s,i} = \frac{|S^{n-1}|}{4n(1-s)} =
  \frac{\omega_n}{4(1-s)}$, where $\omega_n =
  \frac{2\pi^{n/2}}{n\Gamma(n/2)}$ is the volume of unit ball in
  $\mathbb{R}^n$. Another example is the $n$-dimensional cube
  centered at $x_i$ with scale $H_i/\sqrt{n}$, namely
$$
\Omega_i = \{x_i + (z_1, \cdots, z_n)~:~ |z_i| < \frac{H_i}{\sqrt{n}}\}.
$$ 
Clearly, $\Omega_i \subset B_{H_i}(x_i) \subset \Omega$. A direct
calculation shows that in 2D,
\begin{equation} \label{eq:2d-kappa}
\begin{aligned}
\kappa_{n,s,i} &= \frac{1}{1-s} \int_0^{\pi/4}
  \left(\frac{1}{\sqrt{2}\cos\theta}\right)^{2-2s} \,\mathrm{d}\theta
  \\
& = 
\left\{
\begin{array}{ll}
\frac{1}{(1-s)2^{2-s}} \left[ \frac{\sqrt{\pi}\Gamma(s-\frac12)}{\Gamma(s)} - B_{\frac12}(s-\frac12,\frac12) \right] & \quad s\neq 0.5, \\
\frac{1}{(1-s)2^{-s}} \tanh^{-1}\left( \tan(\frac{\pi}{8})\right)  & \quad s = 0.5.
\end{array}
\right.
\end{aligned}
\end{equation}
Here, $B_x(a,b)$ is the incomplete beta function.
\end{remark}

\subsection{Monotonicity}
We seek $u_h \in \mathbb{V}_h$ such that for $x_i \in \mathcal{N}_h^0$,
\begin{equation} \label{eq:Lh}
\mathcal{L}_h[u_h](x_i) := \underbrace{-\kappa_{n,s,i} \frac{\Delta_{FD} u_h(x_i;H_i)}{H_i^{2s}}}_{\mathcal{L}_h^S[u_h](x_i)} 
+ \underbrace{\int_{\Omega_i^c} \frac{u_h(x_i) - u_h(y)}{|x_i - y|^{n+2s}} \,\mathrm{d}y}_{:=\mathcal{L}_h^T[u_h](x_i)} 
= f(x_i).
\end{equation}
A similar definition was first proposed by Huang and Oberman
\cite{huang2014numerical} on 1D uniform grids with $H_i = h$. Any
function in $\mathbb{V}_h$ has a pointwise definition and hence
$\mathcal{L}_h$ is well defined. The tail of integral
$\mathcal{L}_h^T[u_h]$ renders integral of piecewise linear functions
outside $\Omega_i$.

We now show that \eqref{eq:Lh} is monotone and prove the discrete
comparison principle.

\begin{lemma}[monotonicity]\label{lm:monotone}
  Let $v_h, w_h \in \mathbb{V}_h$. If $v_h - w_h$ attains a non-negative maximum at
  an interior node $x_i\in \mathcal{N}_h^0$, then $\mathcal{L}_h[v_h](x_i) \geq \mathcal{L}_h[w_h](x_i)$.

\end{lemma}
\begin{proof}
If $v_h - w_h$ attains a non-negative maximum at $x_i\in \mathcal{N}_h^0$, then $v_h(x_i) \geq w_h(x_i)$, and 
\[
v_h(x_i) - v_h(y) \geq w_h(x_i) - w_h(y) \quad \forall y \in \Omega. 
\] 
Then, we have 
\[
\begin{aligned}
\mathcal{L}_h^S[v_h](x_i) = -\kappa_{n,s,i}\frac{\Delta_{FD} v_h(x_i;H_i)}{H_i^{2s}}  &\geq -\kappa_{n,s,i}\frac{\Delta_{FD} w_h(x_i;H_i)}{H_i^{2s}} = \mathcal{L}_h^S[w_h](x_i) , \\
 \mathcal{L}_h^T[v_h](x_i) = \int_{\mathbb{R}^n \setminus \Omega_i} \frac{v_h(x_i) - v_h(y)}{|x_i
  - y|^{n+2s}} \mathrm{d}y	& \geq \int_{\mathbb{R}^n \setminus
  \Omega_i} \frac{w_h(x_i) - w_h(y)}{|x_i - y|^{n+2s}} \mathrm{d}y = \mathcal{L}_h^T[w_h](x_i),
 \end{aligned}
\] 
which implies $\mathcal{L}_h[v_h](x_i) \geq \mathcal{L}_h[w_h](x_i)$, as asserted.
\end{proof}

\begin{lemma}[discrete barrier function]\label{lm:barrier}
Let $b_h \in \mathbb{V}_h$ satisfy
	\begin{equation}
		\begin{aligned}
			b_h(x_i) := 1\quad \forall x_i \in \mathcal{N}_h^0.
		\end{aligned}
	\end{equation}
Denoting $\delta_i$ as the shorthand of $\delta(x_i)$, then we have
	\begin{equation}
		\mathcal{L}_h[b_h](x_i) \geq C\delta_i^{-2s}\quad \forall x_i \in \mathcal{N}_h^0,
	\end{equation}
where the constant $C$ depends only on $s$ and $\Omega$.
\end{lemma}
\begin{proof}
Since $\mathcal{L}_h^S[b_h] \geq 0$, it suffices to prove $\mathcal{L}_h^T[b_h](x_i)\geq C\delta_i^{-2s}$.
By the definition of $b_h$, 
\[
\begin{aligned}
	\mathcal{L}_h^T[b_h](x_i) &= \int_{\mathbb{R}^n \setminus \Omega_i} \frac{b_h(x_i) - b_h(y)}{|x_i - y|^{n+2s}} \mathrm{d}y
	\geq \int_{\Omega^c}\frac{b_h(x_i) - b_h(y)}{|x_i - y|^{n+2s}} \mathrm{d}y\\
    & = \int_{\Omega^c}\frac{1}{|x_i - y|^{n+2s}} \mathrm{d}y \geq
    C\delta_i^{-2s},
	\end{aligned}
	\]
  where \eqref{eq:kernel-int} is used in the last step. 
\end{proof}

\begin{lemma}[discrete comparison principle]\label{lm:dcp}
Let $v_h, w_h\in \mathbb{V}_h$ be such that
\begin{equation} \label{eq:dcp}
	\mathcal{L}_h[v_h](x_i)\geq \mathcal{L}_h[w_h](x_i) \quad \forall x_i \in \mathcal{N}_h^0.
\end{equation}
Then, $v_h \geq w_h$ in $\Omega$.
\end{lemma}
\begin{proof}
Since $v_h, w_h \in \mathbb{V}_h$, it suffices to prove $v_h(x_i)\geq w_h(x_i)$ for all $x_i \in\mathcal{N}_h^0$.\\
The proof splits into two steps according to whether the inequality \eqref{eq:dcp} is strict or not.

\underline{Case 1: Strict inequality.} That is, $\mathcal{L}_h[v_h](x_i) > \mathcal{L}_h[w_h](x_i), \forall x_i  \in\mathcal{N}_h^0$.
We assume by contradiction that there exists an interior node $x_k\in \mathcal{N}_h^0$ such that $v_h(x_k) < w_h(x_k)$ and 
\[
v_h(x_k) - w_h(x_k) \leq v_h(x_i) - w_h(x_i) \quad \forall x_i \in \mathcal{N}_h. 
\]
Reasoning as in Lemma \ref{lm:monotone} we obtain $\mathcal{L}_h[v_h](x_k) \leq \mathcal{L}_h[w_h](x_k)$, which contradicts to the strict inequality at $x_k$.

\underline{Case 2: Non-strict inequality.} Let $b_h\in \mathbb{V}_h$ be the discrete barrier function defined in Lemma \ref{lm:barrier}, which satisfies
\[
\mathcal{L}_h[b_h](x_i) \geq C\delta_i^{-2s} \geq C_0,
\] 
where $C_0$ is a fixed positive constant. For arbitrary $\varepsilon > 0$,
  the function $v_h + \varepsilon b_h$ satisfies $v_h + \varepsilon b_h = w_h$
  on $\partial \Omega$, and
\[
\mathcal{L}_h[v_h+ \varepsilon b_h](x_i) \geq \mathcal{L}_h[v_h](x_i) +
  \varepsilon C_0 > \mathcal{L}_h[w_h](x_i)\quad \forall x_i \in
  \mathcal{N}_h^0.
\]
Applying Case 1 we deduce 
\[
v_h + \varepsilon b_h \geq w_h \quad \forall \varepsilon > 0.
\]
Taking the limit as $\varepsilon \to 0$ leads to the asserted inequality.
\end{proof}

In the following sections, we consider the error estimate of \eqref{eq:Lh}
with the choice
\begin{equation} \label{eq:Hi}
H_i = h_i^\alpha \min\{\delta_i^{1-\alpha}, \delta_0^{1-\alpha}\},
\end{equation}
where $\alpha \in [\frac12,1]$ is a parameter that will be determined later,
$\delta_i$ is the shorthand of $\delta(x_i)$, and $\delta_0 > 0$ is a
fixed constant. The definition of $h_i$ implies that $h_i \leq
\delta_i$, or $H_i \leq h_i^\alpha \delta_i^{1-\alpha} \leq \delta_i$
and hence $B_{H_i}(x_i) \subset \Omega$.

We first divide the consistency error into three components 
\begin{equation} \label{eq:consistency} 
\begin{aligned}
\mathcal{L}_h[\mathcal{I}_h u](x_i) - \mathcal{L}[u](x_i)
& = \mathcal{L}_h^S[u](x_i) - \int_{\Omega_i} \frac{u(x_i) -
  u(y)}{|x_i - y|^{n+2s}} \,\mathrm{d}y \\
& +  \mathcal{L}_h^S[\mathcal{I}_hu](x_i) - \mathcal{L}_h^S[u](x_i) \\
& + \int_{\Omega_i^c} \frac{u(y) - \mathcal{I}_hu(y)}{|x_i -
  y|^{n+2s}} \,\mathrm{d}y,
\end{aligned}
\end{equation}
where $\mathcal{I}_h$ stands for the nodal interpolation.  A simple
observation is that the first component does not entail
$\mathcal{I}_h$, which is estimated without the underlying grids in Section
\ref{sc:Lhs}. We then estimate the other two components respectively using the information of underlying grids in Section \ref{sc:Lht}. 

\section{Consistency error I: Singular part} \label{sc:Lhs}
In this section, we quantify the operator consistency error of
$\mathcal{L}_h^S$ in conjunction with the H\"{o}lder regularity
results. 

\subsection{$\delta$-interior region}
As shown in Corollary \ref{co:rho-holder}, the H\"{o}lder
constant of $u$ at $x$ depends on $\delta(x)$. It is therefore useful
to define the node $x_i$ such that all the points around $x_i$ share a
similar $\delta(\cdot)$. Recall that $\lambda$ in
\eqref{eq:shape-reg3} and $\bar{c}_S$ in \eqref{eq:theta-uniformity},
let us define a set of {\it $\delta$-interior nodes} 
\begin{equation} \label{eq:interior-node}
\mathcal{N}_h^{0,\delta} := \left\{ x_i \in \mathcal{N}_h^0 ~:~
  \frac{\delta_i}{h_i} \geq C_\delta \right\}, \quad C_\delta :=
  \max\left\{ \left(2\max\{\bar{c}_S, 1\}\right)^{1/\alpha}, 2\lambda
  \right\} .
\end{equation} 
We also denote 
$$ 
\tilde{\Omega}_i := \Omega_i \cup B_{H_i}(x_i) \subset
\Omega.
$$ 
We have the following lemma.

\begin{lemma}[uniform $\delta$ in $\tilde{\Omega}_i \cup \omega_i$]
  \label{lm:uniform-delta} 
For any $x_i \in \mathcal{N}_h^{0,\delta}$, it holds
$$ 
\frac{1}{2}\delta_i \leq \delta(x) \leq \frac{3}{2} \delta_i \quad
  \forall x \in \tilde{\Omega}_i \cup \omega_i.
$$ 
\end{lemma}
\begin{proof} From the definition of $\mathcal{N}_h^{0,\delta}$ and
shape regularity \eqref{eq:shape-reg3}, we have 
$$
  \begin{aligned}
    \delta_i &= \delta_i^\alpha\delta_i^{1-\alpha} \geq
    2\max\{\bar{c}_S,1\} h_i^\alpha \delta_i^{1-\alpha} \geq
    2\max\{\bar{c}_S,1\}H_i, \\
    \delta_i & \geq 2\lambda h_i \geq 2 \max_{\bar{T} \ni x_i}h_T.
  \end{aligned}
$$ 
Using \eqref{eq:theta-uniformity} (quasi-uniformity condition of
$\Omega_i$), we have 
$$ 
  \begin{aligned}
    |\delta(x) - \delta_i| \leq \max\{\bar{c}_S,1\}H_i &\leq
    \frac{1}{2} \delta_i  \quad \forall x \in \Omega_i \cup
    B_{H_i}(x_i), \\ 
    |\delta(x) - \delta_i| \leq \max_{\bar{T} \ni x_i}h_T &\leq  
    \frac{1}{2}\delta_i \quad \forall x \in \omega_i, 
  \end{aligned}
$$
which completes the proof.
\end{proof}

We define the {\it $\delta$-interior region} and {\it
$\delta$-interior triangulation} as 
\begin{equation} \label{eq:interior-region}
  \Omega_h^{0,\delta} := \bigcup_{x_i \in \mathcal{N}_h^{0,\delta}} \omega_i,
  \quad 
  \mathcal{T}_h^{0,\delta} := \{T \in \mathcal{T}_h:~\exists x_i \in
  \mathcal{N}_h^{0,\delta} \mbox{such that } x_i \in \bar{T}\}.
\end{equation}
As a direct consequence of Lemma \ref{lm:uniform-delta}, for any $T
\in \mathcal{T}_h^{0,\delta}$, it holds that $\frac13 \delta(x') \leq
\delta(x) \leq 3 \delta(x')$ for any $x,x' \in T$.

\subsection{Consistency of $\mathcal{L}_h^S$}
\commentone{Thanks to} Corollary \ref{co:rho-holder} ($\delta$-dependence in
H\"{o}lder norm), we have the following lemmas.
\begin{lemma}[$\delta$-interior consistency of $\mathcal{L}_h^S$]
  \label{lm:consistency-s1}
Let $\hat{\beta} := \min\{\beta, 4\}$, and $\beta > 2s$ be such that neither $\hat{\beta} - 2s$ nor $\hat{\beta}$ is an integer.  
Let $f \in L^\infty(\Omega)$ be such that 
$\|f\|_{\hat{\beta}-2s;\Omega}^{(s)} < \infty$.
Then, the solution of \eqref{eq:fL} satisfies
\begin{equation} \label{eq:consistency-s1}
\left| 
\mathcal{L}_h^S[u](x_i) - \int_{\Omega_i} \frac{u(x_i) - u(y)}{|x_i -
  y|^{n+2s}} \,\mathrm{d}y \right| \leq C(\hat{\beta}) \delta_i^{s - \hat{\beta} }
  H_i^{\hat{\beta} - 2s} \quad \forall x_i \in \mathcal{N}_h^{0,\delta},
\end{equation}
where the constant, with emphasis of the dependence on $\hat{\beta}$, also depends on $\Omega$, $s$,
$\|f\|_{L^\infty(\Omega)}$ and $\|f\|_{\hat{\beta}-2s;\Omega}^{(s)}$.
\end{lemma}
\begin{proof}
We consider the proof in two cases.

\underline{Case 1: $\beta \leq 2$.} In this case, $\hat{\beta} = \beta$. Since $\kappa_{n,s,i}$ is bounded,
then using Lemma \ref{lm:fd-holder} (FD in H\"{o}lder seminorm), we have
$$ 
\begin{aligned}
  \left|\mathcal{L}_h^S[u](x_i)\right| & \leq \kappa_{n,s,i}
  H_i^{-2s} \sum_{j=1}^n \left| 2u(x_i) - u(x_i+H_ie_j) -
  u(x_i-H_ie_j) \right| \\
	& \leq C|u|_{C^\beta (B_{H_i}(x_i))} H_i^{\beta - 2s}.
\end{aligned}
$$ 
Using \eqref{eq:theta-uniformity}, \eqref{eq:reflective-sym} and Lemma \ref{lm:fd-holder}, we obtain 
$$ 
\begin{aligned}
\left| \int_{\Omega_i} \frac{u(x_i) - u(y)}{|x_i - y|^{n+2s}}
  \,\mathrm{d}y \right|
  &=  \frac{1}{2}\left|\int_{\Omega_i - x_i} \frac{2u(x_i) - u(x_i+z)
  - u(x_i-z)}{|z|^{n+2s}} \,\mathrm{d}z\right| \\
  &\leq C |u|_{C^\beta( \Omega_i)} \int_{\Omega_i - x_i}
  \frac{|z|^\beta}{|z|^{n+2s}} \,\mathrm{d}z  \\
	&= C|u|_{C^\beta(\Omega_i)}  \commentone{\int_{S^{n-1}} \int_{0}^{\rho_i(\theta)}} \rho^{\beta - 1 - 2s}\,\mathrm{d}\rho \commentone{\,\mathrm{d}S_\theta}\\
  &= C \frac{H_i^{\beta-2s}}{\beta-2s}  |u|_{C^\beta(\Omega_i)}
  \int_{S^{n-1}} \left(\frac{\rho_i(\theta)}{H_i}\right)^{\beta - 2s}
  \,\mathrm{d}S_\theta \\ 
  &\leq C|u|_{C^\beta (\Omega_i)} H_i^{\beta -
  2s}.
		\end{aligned}
		$$ 
By Lemma \ref{lm:uniform-delta} (uniform $\delta$ in $\tilde{\Omega}_i
\cup \omega_i$) and Corollary \ref{co:rho-holder} ($\delta$-dependence
in H\"{o}lder norm), we have that for any $\delta$-interior node $x_i \in
  \mathcal{N}_h^{0,\delta}$
$$
\max\left(|u|_{C^\beta(\Omega_i)}, |u|_{C^\beta(B_{H_i}(x_i))} \right)
  \leq |u|_{C^\beta(\tilde{\Omega}_i)} \leq C(\beta) \delta_i^{s-\beta},
$$
which leads to the desired result by combining the above two estimates. 

\underline{Case 2: $\beta > 2$.} By triangle inequality, we have
\[
\begin{aligned}
\left|\mathcal{L}_h^S[u](x_i) - \int_{\Omega_i} \frac{u(x_i) -
  u(y)}{|x_i - y|^{n+2s}}\; \mathrm{d} y   \right|\leq
  \left|\mathcal{L}_h^S[u](x_i)+\kappa_{n,s,i}\Delta u(x_i)H_i^{2-2s}
  \right|\\
+\left|\kappa_{n,s,i}\Delta u(x_i)H_i^{2-2s}  +\int_{\Omega_i}
  \frac{u(x_i) - u(y)}{|x_i - y|^{n+2s}} \mathrm{d} y  \right| := E_1
  + E_2.
\end{aligned}
\]
Thanks to Lemma \ref{lm:fd-holder} and Corollary \ref{co:rho-holder}, the estimate of $E_1$ is
standard:
\[
  \begin{aligned}
  E_1
  &= \kappa_{n,s,i}\left| \frac{\Delta_{\rm FD}
  u(x_i;H_i)}{H_i^{2}} - \Delta u(x_i)  
  \right|H_i^{2-2s}\\
  & \leq C|u|_{C^{\hat{\beta}}(B_{H_i}(x_i))} H_i^{\hat{\beta} -2s}
  \leq C(\hat{\beta}) \delta_i^{s-\hat{\beta}} H_i^{\hat{\beta} -2s}\quad \forall x_i \in
  \mathcal{N}_h^{0,\delta}.		
	\end{aligned}
\]
Using the definition of $\kappa_{n,s,i}$ in \eqref{eq:integral-symmetry}, 
$$ 
\begin{aligned}
  \kappa_{n,s,i}\Delta u(x_i)H_i^{2-2s} &= \frac{1}{4(1-s)} D^2u(x_i):
  \int_{S^{n-1}} \rho_i^{2-2s}(\theta) \theta \otimes \theta
  \,\mathrm{d}S_\theta \\
  &= \frac{1}{2} \commentone{ \int_{S^{n-1}}\int_0^{\rho_i(\theta)}} 
  \rho^{1-2s} D^2u(x_i):\theta \otimes \theta \,\commentone{ \mathrm{d}\rho}\,\mathrm{d}S_\theta,
\end{aligned}
$$ 
together with \eqref{eq:reflective-sym}, yields  
\[
\begin{aligned}
  E_2 & =  \left| \frac{1}{2}
  \commentone{\int_{S^{n-1}} \int_{0}^{\rho_i(\theta)}} \frac{2u(x_i) - u(x_i + \rho \theta ) -u(x_i -\rho
  \theta)+\frac{\partial^2}{\partial \theta^2}u(x_i) \rho
  ^2}{\rho^{1+2s}} \;\commentone{\mathrm{d}\rho}\; \mathrm{d}S_\theta \right|\\
  & \leq  C\int_{S^{n-1}} \int_{0}^{\rho_i(\theta)}
  |u|_{C^{\hat{\beta}}(B_\rho(x_i))} \rho^{\hat{\beta}-1-2s}\mathrm{d}\rho
  \mathrm{d}S_\theta  \leq  C(\hat{\beta}) \delta_i^{s-\hat{\beta}} H_i^{\hat{\beta} -2s} \quad
  \forall x_i \in \mathcal{N}_h^{0,\delta}.
	\end{aligned}
	\]
This completes the proof.
\end{proof}

\begin{lemma}[global consistency of $\mathcal{L}_h^S$] 
\label{lm:consistency-s2}
Let $\beta > 2s$ and  $f \in L^\infty(\Omega)$ be such that $\|f\|_{\beta-2s;\Omega}^{(s)}< \infty$.
Then, the solution of \eqref{eq:fL} satisfies
\begin{equation} \label{eq:consistency-s2}
\left| \mathcal{L}_h^S[u](x_i) - \int_{\Omega_i} \frac{u(x_i) -
  u(y)}{|x_i - y|^{n+2s}} \,\mathrm{d}y \right| \leq C \delta_i^{-s}
  \quad \forall x_i \in \mathcal{N}_h^{0}.
\end{equation}
Here, the constant $C$ depends on $\Omega$, $s$, $\beta$,
$\|f\|_{L^\infty(\Omega)}$ and $\|f\|_{\beta-2s;\Omega}^{(s)}$, but will not blow up as $\beta - 2s$ or $\beta$ being an integer.
\end{lemma}
\begin{proof}
We notice that for any $ \rho\geq 0$, it holds that
$\mathrm{dist}\left(B_\rho(x_i),\partial\Omega \right) + \rho \geq
  \delta_i$. Since $\beta > 2s$, there exists $\beta_0 \in (2s,
  \min\{1+s, \beta\})$, such that neither $\beta_0$ nor $\beta_0 - 2s$ is an integer. Then, we have $u \in
  C^{\beta_0}(\tilde{\Omega}_i)$, whence 
\[
\begin{aligned}
\left|\int_{\Omega_i} \frac{u(x_i) - u(y)}{|x_i - y|^{n+2s}}
  \,\mathrm{d}y\right| 
  &= \left| \frac{1}{2}
  \commentone{\int_{S^{n-1}}\int_{0}^{\rho_i(\theta)}} \frac{2u(x_i) - u(x_i+\rho\theta) - u(x_i - \rho
  \theta)}{\rho^{1+2s}} \commentone{\; \mathrm{d}\rho}\; \mathrm{d}S_\theta\right|\\
  & \leq C \int_{0}^{\delta_i} \left|u
  \right|_{C^{\beta_0}(B_\rho(x_i))} \rho^{\beta_0}
  \frac{1}{\rho^{1+2s}} \,\mathrm{d}\rho \quad\quad\quad(\mbox{by
  Lemma } \ref{lm:fd-holder})\\
  & \leq C(\beta_0)\int_{0}^{\delta_i} \left(\delta_i
  -\rho\right)^{s-\beta_0} \rho^{\beta_0-1-2s}\; \mathrm{d} \rho \quad\quad\quad(\mbox{by
  	Corollary } \ref{co:rho-holder})\\
  & = C(\beta_0)\delta_i^{-s} B\left(s-\beta_0+1, \beta_0-2s\right).
  \end{aligned}
\]
Here, $B(a,b)$ is the beta function. Moreover, using Proposition
\ref{pp:global-holder} and Corollary \ref{co:rho-holder}, we take
$\beta_1 \in (2s, \min\{\beta, 2\}]$, such that neither $\beta_1$ nor $\beta_1 - 2s$ is an integer, to obtain 
\[
\begin{aligned}
\left| \mathcal{L}_h^S[u](x_i) \right| 
  &= \kappa_{n,s,i}\left| \frac{\Delta_{\rm FD} u(x_i;H_i)}{H_i^{2}}
  \right|H_i^{2-2s} \\ 
  & \leq C\min\left\{ 
  |u|_{C^s(B_{H_i}(x_i))} H_i^{-s}, |u|_{C^{\beta_1}(B_{H_i}(x_i))}
  H_i^{\beta_1 -2s} \right\}\\
  & \leq \delta_i^{-s}
  \min\left\{C_1\left(\frac{H_i}{\delta_i}\right)^{-s}, 
  C_2(\beta_1)\left(1-\frac{H_i}{\delta_i}\right)^{s-\beta_1} \left(
  \frac{H_i}{\delta_i} \right)^{\beta_1-2s} \right\}  \\
  & \leq \max\{C_1, C_2(\beta_1)\} 2^s \delta_i^{-s}.
\end{aligned}
\]
  Here, the last step is deduced by considering whether
  $\frac{H_i}{\delta_i} \in [1/2,1)$ or $\frac{H_i}{\delta_i} \in
  (0,1/2)$. The proof is thus complete. 
\end{proof}

\section{Consistency error II: interpolation and tail part}
\label{sc:Lht}

In this section, we quantify the consistency errors of interpolation
and $\mathcal{L}_h^T$ in \eqref{eq:consistency} on the following {\it
graded grids} with parameter $h$: There is a number $\mu \geq 1$ such
that for any $T \in \mathcal{T}_h$,
\begin{equation} \label{eq:graded-h}
h_T \eqsim \left\{
\begin{array}{ll}
h^{\mu} & \mbox{if } T \cap \partial \Omega \neq \varnothing, \\
h \mathrm{dist}(T, \partial\Omega)^{\frac{\mu - 1}{\mu}} & \mbox{if }
 T \cap \partial \Omega = \varnothing,
\end{array}
\right.
\end{equation}
\commentone{where $X \lesssim Y$ means that  there exists a constants $C> 0$ such that $X \leq C Y$, and $X \eqsim Y$ means that $X \lesssim Y$ and $Y \lesssim X$.} 

As a direct consequence, 
\begin{equation} \label{eq:graded-lam0}
\lambda_0^{-1} h^\mu \leq h_T \quad \forall T \in \mathcal{T}_h.
\end{equation}

We note that the quasi-uniform
grid corresponds to the case in which $\mu = 1$, thus the results in
this section are applicable to the quasi-uniform grids. Further, the
construction \eqref{eq:graded-h} yields a total number of degrees of
freedom (cf.
\cite{babuvska1979direct}) 
\begin{equation} \label{eq:graded-dof}
N := \mathrm{dim} \mathbb{V}_h \eqsim \left\{
\begin{array}{ll}
h^{-n} & \mbox{if } \mu < \frac{n}{n-1}, \\
 |\log h| h^{-n} & \mbox{if } \mu = \frac{n}{n-1}, \\
h^{(1-n)\mu} & \mbox{if } \mu > \frac{n}{n-1}.
\end{array}
\right.
\end{equation}

\subsection{Characterization of $\delta$-interior region}
We begin with the characterization of the $\delta$-interior region
$\Omega_h^{0,\delta}$ defined in \eqref{eq:interior-region}.  The
first observation is that any point outside $\Omega_h^{0,\delta}$
should be very close to $\partial \Omega$, which is rigorously stated
in the following lemma. 

\begin{lemma}[$\delta$ outside $\Omega_h^{0,\delta}$]
\label{lm:delta-outside}
On the graded grids, there exists $C_b > 0$ such that  
  \begin{subequations} \label{eq:delta-outside}
\begin{align} 
\mathrm{dist}(T,\partial\Omega) &\leq C_b h^\mu \quad \forall T \in
  \mathcal{T}_h \setminus \mathcal{T}_h^{0,\delta},
  \label{eq:delta-outside1}\\
  \delta(y) & \leq C_b h^\mu \quad \forall y \in \Omega \setminus
  \Omega_h^{0,\delta}. \label{eq:delta-outside2}
\end{align}
\end{subequations}
\end{lemma}
\begin{proof}
For any $T\in \mathcal{T}_h \setminus \mathcal{T}_h^{0,\delta}$, the case in which $T \cap \partial \Omega \neq \varnothing$ is trivial
since the left hand side vanishes. Below we consider the case in which
$T \cap \partial\Omega = \varnothing$.  Let $x_j$ be any vertex of
$T$.  From the definition of $\delta$-interior triangulation
\eqref{eq:interior-region}, we see that $x_j \notin
\mathcal{N}_h^{0,\delta}$, whence from the local quasi-uniformity
\eqref{eq:shape-reg2},
$$ 
\mathrm{dist}(T,\partial\Omega) \leq \delta(x_j) \leq C_\delta h_j
  \leq \lambda_2 C_\delta h_T,
$$ 
where $C_\delta$ is defined in \eqref{eq:interior-node}. Together with
  the graded grid condition \eqref{eq:graded-h}, we obtain 
$$ 
  \mathrm{dist}(T,\partial\Omega) \leq \lambda_0\lambda_2C_\delta h
  \mathrm{dist}(T,\partial\Omega)^{(\mu-1)/\mu},
$$ 
which leads to \eqref{eq:delta-outside1}. Due to \eqref{eq:graded-h}, we
then have $h_T \leq C h^\mu$ for $T\in \mathcal{T}_h \setminus \mathcal{T}_h^{0,\delta}$ and hence \eqref{eq:delta-outside2}.
The proof is thus complete. 
\end{proof}

\begin{lemma}[inside-outside distance]
\label{lm:in-out-distance}
On the graded grids, there exists $C_d > 0$ such that  
\begin{equation} \label{eq:in-out-distance}
 |x_i - y| \geq (1+C_d)^{-1} \delta_i \quad \forall x_i
 \in \mathcal{N}_h^{0,\delta}, y \in \Omega \setminus
 \Omega_h^{0,\delta}.
\end{equation}
\end{lemma}
\begin{proof}
In view of \eqref{eq:shape-reg3} and \eqref{eq:graded-lam0}, we define
$C_d := \lambda_0\lambda C_b$. Next, we prove the assertion by a
classified discussion on $\delta_i$.  Specifically, if $\delta_i \geq
  (1+C_d)(\lambda_0\lambda)^{-1}h^{\mu}$, then by
\eqref{eq:delta-outside2},
$$ 
|x_i - y| \geq \delta_i - \delta(y) \geq \delta_i - C_b h^{\mu}
  \geq [1 - (1+C_d)^{-1}C_d] \delta_i = (1+C_d)^{-1} \delta_i.
$$ 
Otherwise, by the definition of $\Omega_h^{0,\delta}$ in
\eqref{eq:interior-region}, the sphere centered at $x_i$ with radius
$h_i$ contains in $\Omega_h^{0,\delta}$. Therefore, using
\eqref{eq:shape-reg3} and \eqref{eq:graded-lam0},
  $$ 
  |x_i - y| \geq h_i \geq \lambda^{-1} \lambda_0^{-1} h^{\mu} >
  (1+C_d)^{-1} \delta_i.
  $$ 
  This finishes the proof.
\end{proof}

\subsection{Consistency of interpolation}
Following the standard polynomial approximation theory \cite[Section
4]{brenner2008mathematical}, for any $T \in \mathcal{T}_h$ and $v \in
C^\beta(T)$ ($\beta \geq 0$), it holds that
\begin{equation} \label{eq:approx}
  \|v - \mathcal{I}_h v\|_{L^\infty(T)} \leq C
  \|v\|_{C^{\tilde{\beta}}(T)} h_T^{\tilde{\beta}}, \quad
  \tilde{\beta}:= \min\{\beta,2\}.
\end{equation}
\commentone{We recall the symbols} \commenttwo{$\hat{\beta} := \min\{\beta,4\}$ and $\tilde{\beta} := \min\{\beta,2\}$}, \commentone{which will be used many times in this Section.} 

We first exploit the approximation error in the $\delta$-interior region.
Thanks to Lemma \ref{lm:uniform-delta} (uniform $\delta$ in
$\tilde{\Omega}_i \cup \omega_i$) and Corollary \ref{co:rho-holder}
($\delta$-dependence in H\"{o}lder norm), we conclude from
\eqref{eq:approx} that, if neither $\tilde{\beta}-2s$ nor $\tilde{\beta}$ is
an integer, the solution of \eqref{eq:fL} satisfies  
\begin{equation} \label{eq:in-approx}
  |u(y) - \mathcal{I}_h u(y)| \leq C(\tilde{\beta}) \delta(y)^{s -
  \tilde{\beta}} (h\delta(y)^{1 - \frac{1}{\mu}})^{\tilde{\beta}} =
  C(\tilde{\beta}) h^{\tilde{\beta}}\delta(y)^{s -
  \frac{\tilde{\beta}}{\mu}} \quad \forall y \in
  \Omega_h^{0,\delta},
\end{equation}
where the constant depends on $\Omega$, $s$, $\tilde{\beta}$,
$\|f\|_{L^\infty(\Omega)}$ and $\|f\|_{\tilde{\beta}
-2s;\Omega}^{(s)}$.  \commentone{In case that $s - \tilde{\beta}/\mu < 0$, we can take $\tilde{\beta}_1 \in (\mu s, \tilde{\beta}]$, such that} \commenttwo{neither} \commentone{$\tilde{\beta}_1$ nor $\tilde{\beta}_1 - 2s$ is an integer, then the fact $h^{\mu} \lesssim \delta(y)$ for all $y \in \Omega_h^{0,\delta}$ leads to $|u(y) - \mathcal{I}_hu(y)| \leq C(\tilde{\beta}_1) h^{\tilde{\beta}_1} \delta(y)^{s -
  \frac{\tilde{\beta}_1}{\mu}} \leq C h^{\mu s}$.} On the other hand, the global $C^s$-H\"{o}lder
regularity (Proposition \ref{pp:global-holder}) together with Lemma
\ref{lm:delta-outside} ($\delta$ outside $\Omega_h^{0,\delta}$) yield 
\begin{equation} \label{eq:out-approx}
\|u - \mathcal{I}_h u\|_{L^\infty(T)} \leq C h^{\mu s} \quad \forall T
  \in \mathcal{T}_h \setminus \mathcal{T}_h^{0,\delta},
\end{equation}
where $C$ depends on $\Omega$, $s$ and $\|f\|_{L^\infty(\Omega)}$. As
a consequence of \eqref{eq:in-approx}, \eqref{eq:out-approx}, we have that \commentone{for any $\mu \geq 1$},
\begin{equation} \label{eq:global-approx}
  \|u - \mathcal{I}_h u\|_{L^\infty(\Omega)} \leq \max\{Ch^{\mu s},
  C(\tilde{\beta}) h^{\tilde{\beta}}\}, 
 \end{equation}
where it is assumed again that neither $\tilde{\beta}-2s$ nor $\tilde{\beta}$
is an integer.

We are now in the position to discuss the consistency of interpolation
in \eqref{eq:consistency}.
\begin{lemma}[consistency of interpolation] \label{lm:consistency-Ih}
Let $\beta > 2s$ be such that neither $\tilde{\beta} - 2s$ nor $\tilde{\beta}$
  is an integer. Let $f \in L^\infty(\Omega)$ be such that
  $\|f\|_{\tilde{\beta}-2s;\Omega}^{(s)} <
  \infty$.  On the graded grids \eqref{eq:graded-h} \commentone{with any $\mu \geq 1$}, the solution of \eqref{eq:fL}
  satisfies
\begin{subequations}
\begin{align}
  \left|\mathcal{L}_h^S[\mathcal{I}_hu](x_i) - \mathcal{L}_h^S[u](x_i)
  \right| &\leq C(\tilde{\beta})
  h^{\tilde{\beta}}\delta_i^{s-\frac{\tilde{\beta}}{\mu}}
  H_i^{-2s} \qquad\qquad \forall x_i \in \mathcal{N}_h^{0,\delta},
  \label{eq:consistency-Ih1}   \\
  \left|\mathcal{L}_h^S[\mathcal{I}_hu](x_i) -
  \mathcal{L}_h^S[u](x_i) \right| &\leq \max\{Ch^{\mu s},
  C(\tilde{\beta}) h^{\tilde{\beta}}\} H_i^{-2s} \quad \forall x_i \in
  \mathcal{N}_h^{0}, \label{eq:consistency-Ih2}
\end{align}
\end{subequations}
where the constant $C(\tilde{\beta})$ depends also on $\Omega$, $s$,
  $\|f\|_{L^\infty(\Omega)}$ and $\|f\|_{\tilde{\beta}
  -2s;\Omega}^{(s)}$.
\end{lemma}
\begin{proof}
Thanks to Lemma \ref{lm:uniform-delta} 
and \ref{eq:in-approx}, we have for any $x_i \in
\mathcal{N}_h^{0,\delta}$, 
	\[
	\begin{aligned}
  \left|\mathcal{L}_h^S [\mathcal{I}_hu-u](x_i)
  \right|&=\left|\kappa_{n,s,i}H_i^{-2s} 
  \Delta_{FD}(\mathcal{I}_hu-u)(x_i;H_i)\right| \\
  &\leq C\|u - \mathcal{I}_h u\|_{L^\infty(B_{H_i}(x_i))} H^{-2s}_i
    \leq C(\tilde{\beta})
    h^{\tilde{\beta}}\delta_i^{s-\frac{\tilde{\beta}}{\mu}} H_i^{-2s},
  \end{aligned}
	\]
  which gives \eqref{eq:consistency-Ih1}. Similarly, the global
  approximation result \eqref{eq:global-approx} leads to
  \eqref{eq:consistency-Ih2}.
\end{proof}

\subsection{Consistency of $\mathcal{L}_h^T$}
We discuss the consistency of the tail integral $\mathcal{L}_h^T$,
which is also considered in two cases: $\mathcal{N}_h^{0}$ and
$\mathcal{N}_h^{0,\delta}$. We first consider the former case.  

\begin{lemma}[global consistency of $\mathcal{L}_h^T$]
\label{lm:consistency-t2}
Let $\beta > 2s$ be such that neither $\tilde{\beta} - 2s$ nor $\tilde{\beta}$
  is an integer. Let $f \in L^\infty(\Omega)$ be such that
  $ \|f\|_{\tilde{\beta}-2s;\Omega}^{(s)} <
  \infty$.  On the graded grids, the solution of \eqref{eq:fL}
  satisfies
  \begin{equation} \label{eq:consistency-t2}
\left| 
  \int_{\Omega_i^c} \frac{u(y) - \mathcal{I}_hu(y)}{|x_i -y |^{n+2s}}
    \,\mathrm{d}y \right| \leq \max\{Ch^{\mu s}, C(\tilde{\beta})
    h^{\tilde{\beta}}\} H_i^{-2s}\quad \forall
  x_i \in \mathcal{N}_h^0,
\end{equation}
where the constant $C(\tilde{\beta})$ depends also on $\Omega$, $s$,
  $\|f\|_{L^\infty(\Omega)}$ and $\|f\|_{\tilde{\beta}
  -2s;\Omega}^{(s)}$.
\end{lemma}
\begin{proof}
Using the global approximation result
  \eqref{eq:global-approx}, we obtain for all $x_i \in
  \mathcal{N}_h^0$,
	\[
	\begin{aligned}
  \left| \int_{\Omega_i^c} \frac{u(y) - \mathcal{I}_hu(y)}{|x_i -y |^{n+2s}}
  \,\mathrm{d}y \right|
  & \leq \max\{Ch^{\mu s}, C(\tilde{\beta}) h^{\tilde{\beta}}\}
    \int_{\Omega_i^c} \frac{1}{|x_i -y|^{n+2s}} \mathrm{d}y\\
    &\leq \max\{Ch^{\mu s}, C(\tilde{\beta}) h^{\tilde{\beta}}\}
    \commentone{\int_{S^{n-1}} \int_{\rho_i(\theta)}^{+\infty}} \frac{1}{\rho^{1+2s}}
    \mathrm{d}\rho \commentone{\; \mathrm{d} S_\theta} \\
    & \leq \max\{Ch^{\mu s}, C(\tilde{\beta}) h^{\tilde{\beta}}\} H_i^{-2s},
	\end{aligned}
	\]
  where the quasi-uniformity on
  $\rho_i(\theta)$ \eqref{eq:theta-uniformity} is used in the last step.
\end{proof}

The consistency on the $\delta$-interior nodes would be more delicate.
Let $x_i \in \mathcal{N}_h^{0,\delta}$.  Since $u - \mathcal{I}_h u$
vanishes on $\Omega^c$, we confine the domain of integration of
$\mathcal{L}_h^T$ in \eqref{eq:consistency} as $\Omega_i^c \cap
\Omega$, which can be partitioned into three subdomains (see Figure
\ref{fg:domain}) 
\begin{equation} \label{eq:r123-uniform}
\begin{aligned}
  R_{1,i} &: = \Omega_i^c \cap (\Omega \setminus \Omega_h^{0,\delta}),\\
  R_{2,i} &:= \Omega_i^c \cap \left(\Omega_h^{0,\delta} \cap
  B_{(1+C_d)^{-1}\delta_i}(x_i) \right),\\
  R_{3,i} &:= \Omega_i^c \cap \left(\Omega_h^{0,\delta} \setminus
  B_{(1+C_d)^{-1}\delta_i}(x_i) \right).
\end{aligned}
\end{equation}

\begin{figure}[!htbp]
\centering 
\captionsetup{justification=centering}
\begin{minipage}[b]{0.45\textwidth}
\subfloat[subdomains]{
  \includegraphics[width=0.9\textwidth]{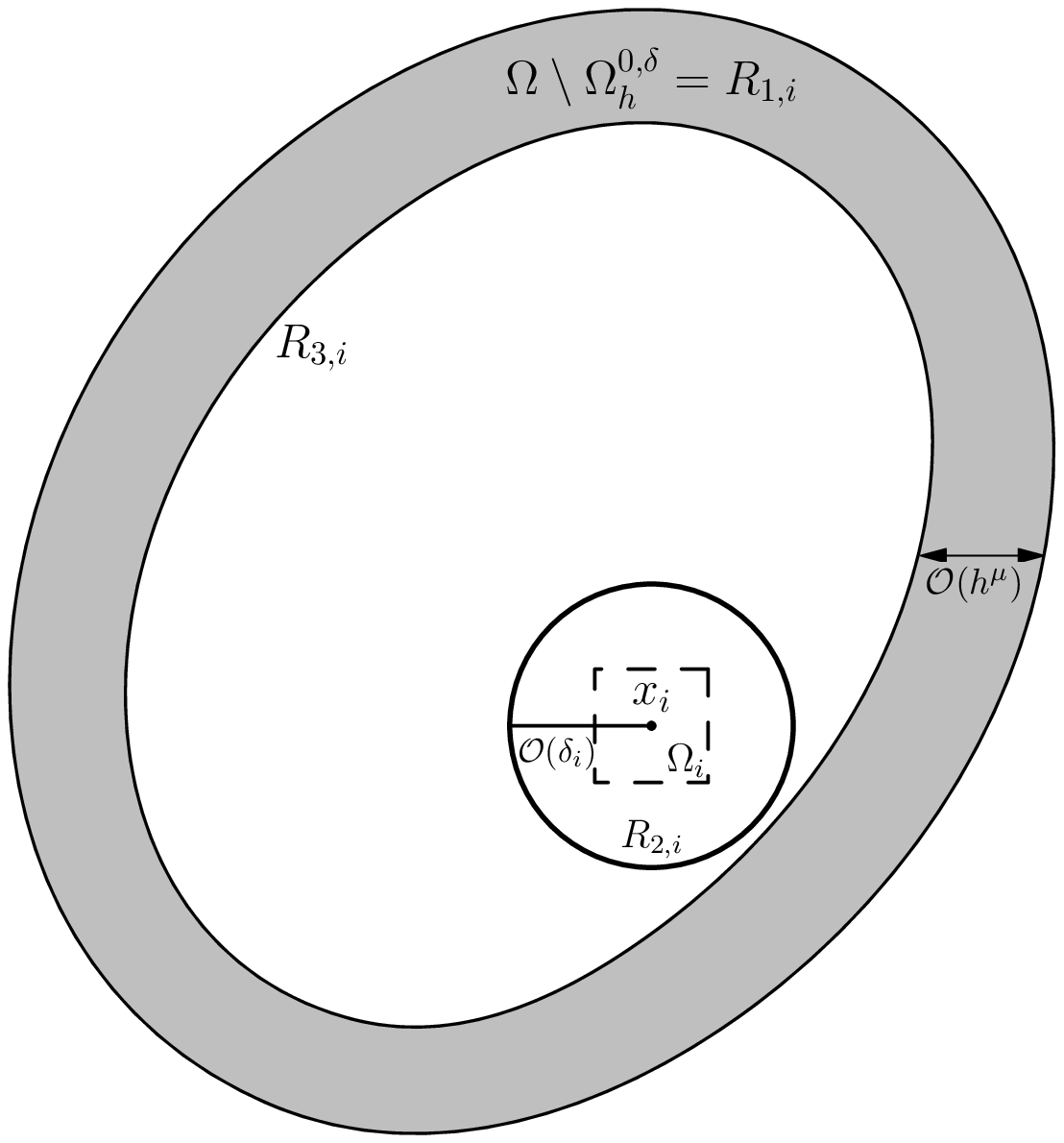}
  \label{fg:domain}
}
\end{minipage}
\begin{minipage}[b]{0.5\textwidth}
  \subfloat[$R_{1,i}$ when $\Omega = \mathbb{R}_+^n$]{
  \includegraphics[width=0.9\textwidth]{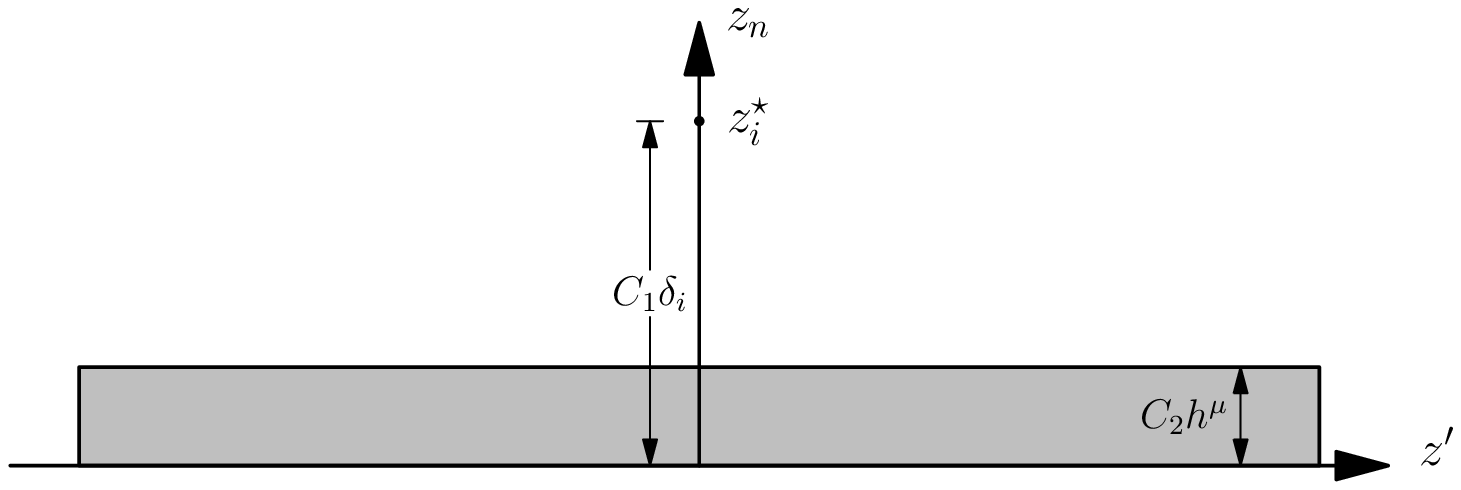}
  \label{fg:R1}
  }

\subfloat[$R_{3,i}$ when $\Omega = \mathbb{R}_+^n$]{
  \includegraphics[width=0.9\textwidth]{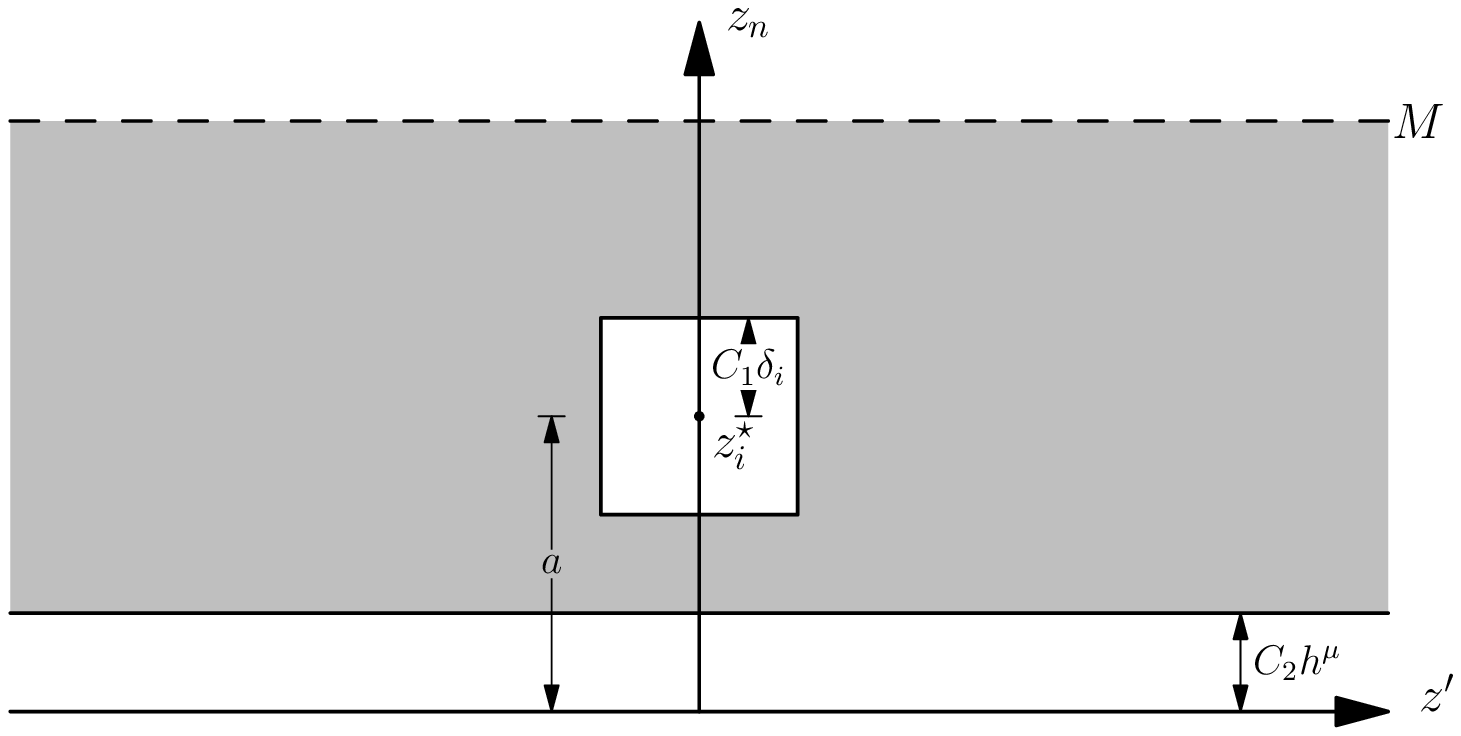}
  \label{fg:R3}
}
\end{minipage}
\caption{Three subdomains of $\Omega_i^c \cap \Omega$. The proofs of Lemma \ref{lm:r1} and Lemma \ref{lm:r3} use (b) and (c), respectively. }
\label{fg:Lh}
\end{figure}

\begin{lemma}[consistency of $\mathcal{L}_h^T$ on $R_{1,i}$] 
\label{lm:r1}
On the graded grids \eqref{eq:graded-h} \commentone{with any $\mu \geq 1$}, there exists $h_0>0$ such that when $h < h_0$,
the solution of \eqref{eq:fL} satisfies
\begin{equation} \label{eq:r1}
  \left| \int_{R_{1,i}} \frac{u(y) - \mathcal{I}_hu(y)}{|x_i -y
  |^{n+2s}} \,\mathrm{d}y \right| \leq C h^{\mu(s+1)} \delta_i^{-1-2s}
  \quad \forall x_i \in \mathcal{N}_h^{0,\delta},
\end{equation}
where the constant depends on $\Omega$, $s$, $\|f\|_{L^\infty(\Omega)}$.
\end{lemma}

\begin{proof}
We apply the approximation result outside $\Omega_h^{0,\delta}$
  \eqref{eq:out-approx} to obtain
$$ 
\begin{aligned}
  \left| \int_{R_{1,i}} \frac{u(y) - \mathcal{I}_hu(y)}{|x_i -y |^{n+2s}}
  \,\mathrm{d}y \right| \leq Ch^{\mu s} \int_{R_{1,i}} \frac{1}{|x_i -
  y|^{n+2s}} \,\mathrm{d}y.
\end{aligned}
$$
Then, the desired result \eqref{eq:r1} can be obtained by
showing  
\begin{equation} \label{eq:r1-estimate}
  \int_{R_{1,i}} \frac{1}{|x_i - y|^{n+2s}} \,\mathrm{d}y \leq C
  h^{\mu} \delta_i^{-1-2s}.
\end{equation}
  Let us consider first the case when $\Omega = \mathbb{R}^n_+ :=
  \{(z', z_n): z' \in \mathbb{R}^{n-1}, z_n > 0\}$.  Invoking Lemma
  \ref{lm:delta-outside} ($\delta$ outside
  $\Omega_h^{0,\delta}$), we consider $z_{i}^\star := (0,\cdots,0,
  C_1\delta_i)$ and $R_{1,i} \subset  \{(z', z_n) : 0 \leq z_n \leq
  C_2h^{\mu} \}$ (see Figure \ref{fg:R1}). We also have $C_1\delta_i -
  C_2h^{\mu} \geq C_3\delta_i$ in light of Lemma
  \ref{lm:in-out-distance} (inside-outside distance).  Then, the
  integral in \eqref{eq:r1-estimate} has the estimate
  $$ 
  \begin{aligned}
    & \int_0^{C_2h^\mu} \,\mathrm{d}z_n \int_{\mathbb{R}^{n-1}}
    \frac{1}{|z_{i}^\star - z|^{n+2s}}\,\mathrm{d}z' \\
    \leq & \int_0^{C_2h^\mu}\,\mathrm{d}z_n \int_{|z'| > \delta_i}
    \frac{1}{|z'|^{n+2s}}\,\mathrm{d}z' + \int_0^{C_2h^\mu}
    \frac{1}{|C_1\delta_i - z_n|^{n+2s}} \,\mathrm{d}z_n
    \int_{|z'|\leq \delta_i} \,\mathrm{d}z' \\
    \leq &\,Ch^\mu \delta_i^{-1-2s} + C[(C_1\delta_i - C_2h^\mu)^{1-n-2s} -
    (C_1\delta_i)^{1-n-2s}]\delta_i^{n-1} \leq Ch^\mu \delta_i^{-1-2s},
  \end{aligned}
  $$ 
  where we use $(1 - \frac{C_2h^\mu}{C_1\delta_i})^{1-n-2s} - 1 \lesssim
  \frac{h^\mu}{\delta_i}$ in the last step. 

  We conclude by extending this result to a general bounded Lipschitz
  domain $\Omega$. We use the notation as in Definition
  \ref{df:Lipschitz} and confine the integral domain to $V\cap
  R_{1,i}$, the general case is then applied by a standard partition
  of unity. Define a point $\tilde{z}_{i}^\star :=
  (0,\cdots,0,\tilde{C}_1\delta_i)$ in the $z$-coordinate. Thanks
  again to Lemma \ref{lm:in-out-distance}, for any $y \in V \cap
  R_{1,i}$ (denoted by $z$ for change of coordinates), $|x_i - y| \geq
  (1+C_d)^{-1}\delta_i \gtrsim |\tilde{z}_{i}^\star - z|$.
  Therefore, by bi-Lipschitz changes of coordinates on $z$, we have  
  $$ 
  \int_{V\cap R_{1,i}} \frac{1}{|x_i - y|^{n+2s}} \,\mathrm{d}y \leq C
  \int_{V\cap R_{1,i}} \frac{1}{|\tilde{z}_{i}^\star - z|^{n+2s}}
  \mathrm{d}z \leq Ch^\mu \delta_{i}^{-1-2s}.
  $$ 
Here, the constant $C$ depends on the Lipschitz constant of $\partial
\Omega$, which is uniformly bounded due to the partition of unity.
This proves \eqref{eq:r1-estimate} and thus \eqref{eq:r1}. 
\end{proof}

\begin{lemma}[consistency of $\mathcal{L}_h^T$ on $R_{2,i}$] \label{lm:r2}
Let $\beta > 2s$ be such that neither $\tilde{\beta} - 2s$ nor $\tilde{\beta}$
  is an integer. Let $f \in L^\infty(\Omega)$ be such that
  $\|f\|_{\tilde{\beta}-2s;\Omega}^{(s)} <
  \infty$.  On the graded grids \eqref{eq:graded-h} \commentone{with any $\mu \geq 1$}, the solution of \eqref{eq:fL}
  satisfies 
\begin{equation} \label{eq:r2}
  \left| \int_{R_{2,i}} \frac{u(y) - \mathcal{I}_hu(y)}{|x_i -y
  |^{n+2s}} \,\mathrm{d}y \right| \leq C(\tilde{\beta})
  h^{\tilde{\beta}} \delta_i^{s - \frac{\tilde{\beta}}{\mu}} H_i^{-2s}
  \quad \forall x_i \in \mathcal{N}_h^{0,\delta},
\end{equation}
where the constant $C(\tilde{\beta})$ depends also on $\Omega$, $s$,
$\|f\|_{L^\infty(\Omega)}$ and $\|f\|_{\tilde{\beta}
-2s;\Omega}^{(s)}$.
\end{lemma}
\begin{proof}
Note that all $y \in R_{2,i}$, the distance satisfies 
$$
\commentone{ \delta_i - |x_i - y| \leq \delta(y) \leq \delta_i + |x_i - y| \leq C\delta_i.}
$$ 

\underline{Case 1: $\mu \leq \tilde{\beta}/s$.} Using \eqref{eq:in-approx}, the
multi-dimensional polar coordinate, and the quasi-uniformity of $\Omega_i$
\eqref{eq:theta-uniformity}, we have
\[
\begin{aligned}
   &\quad \left| \int_{R_{2,i}} \frac{u(y) - \mathcal{I}_hu(y)}{\left|x_i -y
    \right|^{n+2s} }\,\mathrm{d}y\right| 
	  \leq C(\tilde{\beta}) h^{\tilde{\beta}} \int_{R_{2,i}}
    \frac{1}{\delta(y)^{ \frac{\tilde{\beta}}{\mu} - s} \left|x_i -y
    \right|^{n+2s} }\,\mathrm{d}y	\\		
    & \leq C(\tilde{\beta})  h^{\tilde{\beta}} \int_{R_{2,i}}
    \frac{1}{(\delta_i - |x_i - y|)^{\frac{\tilde{\beta}}{\mu}  - s}
    \left|x_i -y  \right|^{n+2s} }\,\mathrm{d}y\\
    &= C(\tilde{\beta})  h^{\tilde{\beta}}\int_{S^{n-1}}
    \int_{\min\{\rho_i(\theta),
    (1+C_d)^{-1}\delta_i\}}^{(1+C_d)^{-1}\delta_i}
    \frac{1}{\left(\delta_i - \rho
    \right)^{\frac{\tilde{\beta}}{\mu}  - s}\rho^{1+2s}}\,\mathrm{d}\rho
    \mathrm{d}S_{\theta}\\
    &\leq C(\tilde{\beta})  h^{\tilde{\beta}}
    \delta_i^{-s-\frac{\tilde{\beta}}{\mu}} \int_{S^{n-1}}
    \int_{\min\{\rho_i(\theta)/\delta_i,
    (1+C_d)^{-1}\}}^{(1+C_d)^{-1}}
    \frac{1}{(1-t)^{\frac{\tilde{\beta}}{\mu}  -
    s}t^{1+2s}}\,\mathrm{d}t \mathrm{d}S_\theta \\
  & \leq C(\tilde{\beta})  h^{\tilde{\beta}} \delta_i^{s -
  \frac{\tilde{\beta}}{\mu}} H_i^{-2s},
	\end{aligned}
\]
where the last step uses the fact that $H_i\leq \delta_i$. 

\commentone{
\underline{Case 2: $\mu > \tilde{\beta}/s$.} In this case, we use the $\delta(y) \leq C\delta_i$ to obtain
\[
\begin{aligned}
   &\quad \left| \int_{R_{2,i}} \frac{u(y) - \mathcal{I}_hu(y)}{\left|x_i -y
    \right|^{n+2s} }\,\mathrm{d}y\right| 
	  \leq C(\tilde{\beta}) h^{\tilde{\beta}} \int_{R_{2,i}}
    \frac{1}{\delta(y)^{ \frac{\tilde{\beta}}{\mu} - s} \left|x_i -y
    \right|^{n+2s} }\,\mathrm{d}y	\\		
    & \leq C(\tilde{\beta})  h^{\tilde{\beta}}  \delta_i^{s - \frac{\tilde{\beta}}{\mu}} \int_{R_{2,i}}
    \frac{1}{
    \left|x_i -y  \right|^{n+2s} }\,\mathrm{d}y
    \leq C(\tilde{\beta})  h^{\tilde{\beta}} \delta_i^{s -
  \frac{\tilde{\beta}}{\mu}} H_i^{-2s},
	\end{aligned}
\]
}
This completes the proof.
\end{proof}

\begin{lemma}[consistency of $\mathcal{L}_h^T$ on $R_{3,i}$] 
\label{lm:r3}
Let $\beta > 2s$ be such that neither $\tilde{\beta} - 2s$ nor $\tilde{\beta}$
  is an integer. Let $f \in L^\infty(\Omega)$ be such that
  $ \|f\|_{\tilde{\beta}-2s;\Omega}^{(s)} <
  \infty$.  On the graded grids \eqref{eq:graded-h} \commentone{with any $\mu \geq 1$}, there exists $h_0>0$ such that when
  $h < h_0$, the solution of \eqref{eq:fL} satisfies, $\forall x_i \in
  \mathcal{N}_h^{0,\delta}$,
\begin{equation} \label{eq:r3}
  \left| \int_{R_{3,i}} \frac{u(y) - \mathcal{I}_hu(y)}{|x_i -y
  |^{n+2s}} \,\mathrm{d}y \right| \leq  
  \left\{
    \begin{array}{ll} 
      C(\tilde{\beta}) h^{\mu(s+1)} \delta_i^{-1-2s} &\mbox{if
      }\tilde{\beta} > \mu(1+s),\\
      C(\tilde{\beta}) h^{\mu(s+1)} |\log h| \delta_i^{-1-2s}
      &\mbox{if }\tilde{\beta} = \mu(1+s),\\
      C(\tilde{\beta}) h^{\tilde{\beta}}
      \delta_i^{-\frac{\tilde{\beta}}{\mu} - s} &\mbox{if } \mu(1+s) >
      \tilde{\beta} > \mu s, \\
      C(\tilde{\beta}) h^{\tilde{\beta}} \delta_i^{-2s} &\mbox{if }
      \mu s \geq \tilde{\beta},
  \end{array}
  \right.
\end{equation}
where the constant $C(\tilde{\beta})$ depends also on $\Omega$, $s$,
$\|f\|_{L^\infty(\Omega)}$ and $\|f\|_{\tilde{\beta}
-2s;\Omega}^{(s)}$. Moreover, the constant $C(\tilde{\beta})$ behaves
  as $|\mu(1+s) - \tilde{\beta}|^{-1}$ when $\tilde{\beta} \to
  \mu(1+s)^-$.
\end{lemma}
\begin{proof}
Since $R_{3,i} \subset \Omega_h^{0,\delta}$, we apply
\eqref{eq:in-approx} to obtain 
$$ 
  \left| \int_{R_{3,i}} \frac{u(y) - \mathcal{I}_hu(y)}{|x_i -y
  |^{n+2s}} \,\mathrm{d}y \right| \leq C(\tilde{\beta}) h^{\tilde{\beta}}
  \int_{R_{3,i}} \frac{1}{\delta(y)^{\frac{\tilde{\beta}}{\mu} - s} |x_i -
  y|^{n+2s}} \,\mathrm{d}y, 
$$ 
which boils down to the estimate of the right integral. Similar to the
estimate of integral on $R_{1,i}$, we consider first the case when
$\Omega = \mathbb{R}_+^n$ but the integral domain is restricted by
$z_n \leq M$ for $M>0$ sufficiently large. Specifically, it holds that
$R_{3,i} \subset \{(z',z_n): z_n \geq C_2h^\mu \}$ (This is because
$C_\delta > 1$ from \eqref{eq:interior-node}, and hence
$\Omega\setminus \Omega_h^{0,\delta}$ has at least one layer of
elements with size $h^\mu$).  Moreover, we consider $z_{i}^\star =
(0,\cdots,0,a)$ and a cylinder $\mathcal{C}_i$ with size
$\delta_i$: (see Figure \ref{fg:R3})
$$
  \mathcal{C}_i := \{ (z',z_n) \,:\, |z_n - a| < C_1\delta_i,~
  |z'| < C_1\delta_i\}. 
$$
Upon the relationship between $\tilde{\beta}$, $\mu(1+s)$ and $\mu s$,
we consider the estimate of the following integral in three cases:
\begin{equation} \label{eq:r3-I}
I:= \int_{\mathbb{R}^{n-1}\times[C_2h^\mu,M] \setminus \mathcal{C}_i}
  \frac{1}{z_n^{\frac{\tilde{\beta}}{\mu}- s} |z_{i}^\star -
  z|^{n+2s}}\,\mathrm{d}z.
\end{equation}

\underline{Case 1: $\tilde{\beta} \geq \mu(1+s)$.} Then, we divide $I$
  into two components:
$$ 
  \begin{aligned}
    I & = \Big( \int_{C_2h^\mu}^M \int_{|z'|>C_1\delta_i} \\
    & \quad + \int_{[C_2h^\mu,M] \setminus [a-C_1\delta_i,
    a+C_1\delta_i]}\int_{|z'|<C_1\delta_i} 
    \Big)
  \frac{1}{z_n^{\frac{\tilde{\beta}}{\mu} - s} |z_{i}^\star -
    z|^{n+2s}}\,\mathrm{d}z'\mathrm{d}z_n \\
  &:= I_1 + I_2.
  \end{aligned}
$$ 
The estimates of $I_1$ and $I_2$ are given as 
$$ 
\begin{aligned}
  I_1 & \leq \int_{C_2h^\mu}^M \frac{1}{z_n^{\frac{\tilde{\beta}}{\mu} - s}}
  \,\mathrm{d}z_n \int_{|z'|>C_1\delta_i}
  \frac{1}{|z'|^{n+2s}}\,\mathrm{d}z' \\
  & \leq C\delta_i^{-1-2s} \int_{C_2h^\mu}^M
  \frac{1}{z_n^{\frac{\tilde{\beta}}{\mu} -
  s}} \,\mathrm{d}z_n \leq 
  \left\{
    \begin{array}{ll}
      Ch^{\mu(s+1) - \tilde{\beta}} \delta_i^{-1-2s} & \mbox{if
      }\tilde{\beta} > \mu(1+s),
      \\
      C |\log h|\delta_i^{-1-2s} & \mbox{if }\tilde{\beta} = \mu(1+s),
    \end{array}
  \right. \\
  I_2 & \leq \int_{[C_2h^\mu,M] \setminus [a-C_1\delta_i,
    a+C_1\delta_i]}
  \frac{1}{z_n^{\frac{\tilde{\beta}}{\mu}-s}|z_n-a|^{n+2s}}\,\mathrm{d}z_n
  \int_{|z'|<C_1\delta_i}\,\mathrm{d}z' \\
  & \leq C \delta_i^{-1-2s} \int_{C_2h^\mu}^{M}
  \frac{1}{z_n^{\frac{\tilde{\beta}}{\mu} - s}} \,\mathrm{d}z_n \leq  
  \left\{
    \begin{array}{ll}
      Ch^{\mu(s+1) - \tilde{\beta}} \delta_i^{-1-2s} & \mbox{if
      }\tilde{\beta} > \mu(1+s), \\
      C |\log h|\delta_i^{-1-2s} & \mbox{if }\tilde{\beta} = \mu(1+s).
    \end{array}
  \right. 
\end{aligned}
$$
Here, the constant blows up as $\tilde{\beta} \to \mu(1+s)^+$.

\underline{Case 2: $\mu(1+s) > \tilde{\beta} >
\mu s$.} In this case, $z_n^{-(\frac{\tilde{\beta}}{\mu} - s)}$ is
unbounded but integrable at $z_n = 0$. We take $\tilde{C}_2 > 0$ so
that \commentone{$\tilde{C}_2 \delta_{i}\geq C_2 h^{\mu}$} (guaranteed by Lemma
\ref{lm:in-out-distance}). Therefore, we divide $I$ into three
components, 
$$
\begin{aligned}
    I & = \Big( \int_{\mathbb{R}^{n-1}\times[\tilde{C}_2 \delta_i,M]
    \setminus \mathcal{C}_i} + \int_{C_2h^\mu}^{\tilde{C}_2\delta_i}
    \int_{|z'|>C_1\delta_i} \\
    & \quad + \int_{[C_2h^\mu,\tilde{C}_2 \delta_i] \setminus [a-C_1\delta_i,
    a+C_1\delta_i]}\int_{|z'|<C_1\delta_i} 
    \Big)
  \frac{1}{z_n^{\frac{\tilde{\beta}}{\mu} - s} |z_{i}^\star -
    z|^{n+2s}}\,\mathrm{d}z'\mathrm{d}z_n \\
  &:= J_0 + J_1 + J_2.
\end{aligned}
 $$
Since $z_n \geq \tilde{C}_2 \delta_i$ in the integral domain of $J_0$, we have
$$
J_0 \leq C \delta_i^{-\frac{\tilde{\beta}}{\mu} + s}
\int_{\mathbb{R}^{n-1}\times[\tilde{C}_2 \delta_i,M] \setminus
\mathcal{C}_i} \frac{1}{|z_i^\star - z|^{n+2s}} \,\mathrm{d}z'\mathrm{d}z_n \leq C
\delta_i^{-\frac{\tilde{\beta}}{\mu} - s}.
$$ 
The estimates of $J_\ell~(\ell = 1,2)$ are similar to the Case 1, i.e., 
$$ 
J_\ell \leq C \delta_i^{-1-2s} \int_{C_2h^\mu}^{\tilde{C}_2 \delta_i}
\frac{1}{z_n^{\frac{\tilde{\beta}}{\mu} - s}} \,\mathrm{d}z_n \leq C
\delta_i^{-\frac{\tilde{\beta}}{\mu} - s} \quad \ell  = 1,2,
$$ 
where the constant blows up as $\tilde{\beta} \to \mu(1+s)^-$.

\underline{Case 3: $\mu s \geq \tilde{\beta}$}. In this case,
$z_n^{-(\frac{\tilde{\beta}}{\mu} - s)}$ is bounded, and hence $I \leq
C(M) \delta_i^{-2s}$.

Combining Case 1-3, we arrive at
\begin{equation} \label{eq:r3-estimate}
  h^{\tilde{\beta}}I \leq 
  \left\{
    \begin{array}{ll} 
      Ch^{\mu(s+1)} \delta_i^{-1-2s} &~ \mbox{if }\tilde{\beta} > \mu(1+s),\\
      C h^{\mu(s+1)} |\log h| \delta_i^{-1-2s} &~ \mbox{if
      }\tilde{\beta} = \mu(1+s),\\
      C h^{\tilde{\beta}} \delta_i^{-\frac{\tilde{\beta}}{\mu} - s} &~
      \mbox{if } \mu(1+s) > \tilde{\beta} > \mu s, \\
      C(M) h^{\tilde{\beta}} \delta_i^{-2s} & ~\mbox{if } \mu s \geq
      \tilde{\beta}.
  \end{array}
  \right. 
\end{equation}

We generalize this result to a bounded Lipschitz domain $\Omega$.
Using the notation as in Definition \ref{df:Lipschitz}, there exists
$\{V_k\}_{k=1}^K$ such that $\partial\Omega \subset \cup_{k=1}^K V_k$.
Notice that $\delta(y) \gtrsim 1$ for $y \in R_{3,i} \setminus
\cup_{k=1}^K V_k$, then 
\begin{equation} \label{eq:r3-1}
\begin{aligned}
  & \quad h^{\tilde{\beta}} \int_{R_{3,i} \setminus \cup_{k=1}^K V_k}
\frac{1}{\delta(y)^{\frac{\tilde{\beta}}{\mu} - s} |x_i - y|^{n+2s}}
  \,\mathrm{d}y \\ 
  & \leq C h^{\tilde{\beta}}\int_{|y - x_i| >
  (1+C_d)^{-1}\delta_i}
  \frac{1}{|x_i - y|^{n+2s}}
  \,\mathrm{d}y \leq Ch^{\tilde{\beta}}\delta_i^{-2s} \leq Ch^{\tilde{\beta}}I.
\end{aligned}
\end{equation} 
For any $V_k\cap R_{3,i}$, using the bi-Lipschitz changes of
coordinates, and noticing that the cylinder $\mathcal{C}_i$ is
equivalent (mutually bounded up to constant) to
$B_{(1+C_d)^{-1}\delta_i}(z_{i}^\star)$, then 
\begin{equation} \label{eq:r3-2}
  \quad h^{\tilde{\beta}} \sum_{k=1}^K \int_{R_{3,i}\cap V_k}
\frac{1}{\delta(y)^{\frac{\tilde{\beta}}{\mu} - s} |x_i - y|^{n+2s}}
\,\mathrm{d}y \leq  Ch^{\tilde{\beta}}I,
\end{equation} 
where the constant $M$ in \eqref{eq:r3-estimate} can be taken as
$\mathcal{O}(\mathrm{diam}(\Omega))$. Combining \eqref{eq:r3-1}
and \eqref{eq:r3-2} yields \eqref{eq:r3}, as asserted. 
\end{proof}

Combining Lemmas \ref{lm:r1}-\ref{lm:r3}, and noticing that $h^\mu
\delta_i^{-1} \lesssim 1$ for the $\delta$-interior nodes, we then
have the $\delta$-interior consistency of $\mathcal{L}_h^T$ as
follows.

\begin{lemma}[$\delta$-interior consistency of $\mathcal{L}_h^T$] 
\label{lm:consistency-t1}
Let $\beta > 2s$ be such that neither $\tilde{\beta} - 2s$ nor $\tilde{\beta}$
  is an integer. Let $f \in L^\infty(\Omega)$ be such that
  $\|f\|_{\tilde{\beta}-2s;\Omega}^{(s)}<
  \infty$.  On the graded grids \eqref{eq:graded-h} \commentone{with any $\mu \geq 1$}, there exists $h_0>0$ such that when
  $h < h_0$, the solution of \eqref{eq:fL} satisfies,
  \begin{equation} \label{eq:consistency-t1}
    \begin{aligned}
      & \left| \int_{\Omega_i^c} \frac{u(y) -
      \mathcal{I}_hu(y)}{|x_i -y |^{n+2s}} \,\mathrm{d}y \right|  
      \leq C(\tilde{\beta}) h^{\tilde{\beta}} \delta_i^{s -
      \frac{\tilde{\beta}}{\mu}} H_i^{-2s} \\ 
      & \qquad  + \left\{ \begin{array}{ll} C(\tilde{\beta})
        h^{\mu(s+1)} \delta_i^{-1-2s} &\mbox{if }\tilde{\beta} >
        \mu(1+s),\\
      C(\tilde{\beta}) h^{\mu(s+1)} |\log h| \delta_i^{-1-2s}
        &\mbox{if }\tilde{\beta} = \mu(1+s),\\
      C(\tilde{\beta}) h^{\tilde{\beta}}
        \delta_i^{-\frac{\tilde{\beta}}{\mu} - s} &\mbox{if } \mu(1+s)
        > \tilde{\beta} > \mu s, \\
      C(\tilde{\beta}) h^{\tilde{\beta}} \delta_i^{-2s} &\mbox{if }
        \mu s \geq \tilde{\beta},
  \end{array}
  \right.
    \quad \forall x_i \in \mathcal{N}_h^{0,\delta}, 
    \end{aligned}
  \end{equation}
where the constant $C(\tilde{\beta})$ depends also on $\Omega$, $s$,
$\|f\|_{L^\infty(\Omega)}$ and $\|f\|_{\tilde{\beta}
-2s;\Omega}^{(s)}$. Moreover, the constant $C(\tilde{\beta})$ behaves
  as $|\mu(1+s) - \tilde{\beta}|^{-1}$ when $\tilde{\beta} \to
  \mu(1+s)^-$.
\end{lemma}

\section{Pointwise error estimate}
\label{sc:error}

In this section, we establish the convergence rate of the proposed scheme
under the choice of $H_i$ given in \eqref{eq:Hi}, i.e.,
$$
H_i = h_i^{\alpha} \min\{\delta_i^{1-\alpha}, \delta_0^{1-\alpha} \}. 
$$
Thanks to Lemma \ref{lm:delta-outside} ($\delta$ outside
$\Omega_h^{0,\delta}$) and \eqref{eq:graded-h}, we have 
\begin{equation} \label{eq:graded-H}
H_i \eqsim \left\{
\begin{array}{ll}
h^{\mu} \eqsim \delta_i & \mbox{if }  x_i \in \mathcal{N}_h^0
  \setminus \mathcal{N}_h^{0,\delta}, \\
h^{\alpha} \delta_i^{1-\frac{\alpha}{\mu}} & \mbox{if } x_i \in
  \mathcal{N}_h^{0,\delta}.
\end{array}
\right.
\end{equation} 
Note that the \eqref{eq:graded-H} also holds for $\delta_i \geq
\delta_0$, where $H_i = h_i^\alpha \delta_0^{1-\alpha} \eqsim h^\alpha
\eqsim h^\alpha \delta_i^{1-\frac{\alpha}{\mu}}$. 

In light of Corollary \ref{co:rho-holder} ($\delta$-dependence in
H\"{o}lder norm) and Remark \ref{rk:blow-up-speed} (blow-up behavior),
we define the following two indices:
\begin{equation}
\begin{aligned}
\hat{\kappa} &:= \left\{
\begin{array}{ll}
1 & \mbox{if } \hat{\beta} - 2s \mbox{ or } \hat{\beta} \mbox{ is an integer}, \\
0 & \mbox{otherwise},
\end{array}
\right. \\
\tilde{\kappa} &:= \left\{
\begin{array}{ll}
1 & \mbox{if } \tilde{\beta} - 2s \mbox{ or } \tilde{\beta} \mbox{ is an integer}, \\
0 & \mbox{otherwise},
\end{array}
\right.
\end{aligned}
\end{equation}
where we recall that $\hat{\beta} = \min\{\beta, 4\}$ and
$\tilde{\beta} = \min\{\beta, 2\}$.

We now derive pointwise error estimates for the solution of
\eqref{eq:fL}. We proceed as follows.  We apply the global and
$\delta$-interior consistency results respectively to control the
consistency errors outside and inside $\Omega_h^{0,\delta}$. The
combination of consistency error and  Lemma \ref{lm:barrier} (discrete
barrier function) will conclude the argument, thanks to Lemma
\ref{lm:dcp} (discrete comparison principle).
 
\begin{theorem}[Convergence rates in terms of $h$] \label{tm:err-h}
Let $\Omega$ be a bounded Lipschitz domain with exterior ball
condition. Let $f \in L^\infty(\Omega)$ be such that   
$\|f\|_{\beta-2s;\Omega}^{(s)}  < \infty$. Then, on the graded 
grids with \commentone{any $\mu \geq 1$} and $H_i$ chosen as in \eqref{eq:Hi}, we have 
\begin{equation} \label{eq:err-h1}
  \|u - u_h\|_{L^\infty(\Omega)} \leq C
  \max\{ h^{\mu s}, |\log h|^{\hat{\kappa}} h^{(\hat{\beta} - 2s)\alpha}, |\log h|^{\tilde{\kappa}} h^{\tilde{\beta} - 2s\alpha} \},
\end{equation}
where the constant depends on $\Omega$, $s$, $\beta$,
$\|f\|_{L^\infty(\Omega)}$ and $\|f\|_{\beta -2s;\Omega}^{(s)}$. 
Moreover, the optimal $\alpha$ and corresponding convergence rate are 
\begin{equation} \label{eq:err-h2}
\alpha_{\mathrm{opt}} := \frac{\tilde{\beta}}{\hat{\beta}} \in
  [\frac12, 1], \quad \|u - u_h\|_{L^\infty(\Omega)} \leq C \max\{
    h^{\mu s}, |\log h|^{\max\{\hat{\kappa}, \tilde{\kappa}\} }
    h^{\tilde{\beta} - 2s\tilde{\beta} / \hat{\beta} } \}.
\end{equation}
\end{theorem}

\begin{proof}
  We consider the consistency error
  $\mathcal{L}_h[\mathcal{I}_hu](x_i) - \mathcal{L}[u](x_i)$ in two
  cases.

\underline{Case 1: $x_i \in \mathcal{N}_h^0 \setminus
  \mathcal{N}_h^{0,\delta}$.} Applying Lemma \ref{lm:consistency-s2}
  (global consistency of $\mathcal{L}_h^S$), Lemma
  \ref{lm:consistency-Ih} (consistency of interpolation) and Lemma
  \ref{lm:consistency-t2} (global consistency of $\mathcal{L}_h^T$),
  we have
 $$
  |\mathcal{L}_h[\mathcal{I}_hu](x_i) - \mathcal{L}[u](x_i)|  \leq C
    \delta_i^{-s} + \max\{Ch^{\mu s},
  C(\tilde{\beta}) h^{\tilde{\beta}}\} H_i^{-2s} \quad
    \forall x_i \in \mathcal{N}_h^0 \setminus
    \mathcal{N}_h^{0,\delta}.
$$ 
If $\tilde{\beta}$ or $\tilde{\beta} - 2s$ is an integer, we use
  Remark \ref{rk:blow-up-speed} (blow-up behavior) to obtain that, for
  arbitrary small $\varepsilon > 0$,
$$ 
C(\tilde{\beta} - \varepsilon) h^{\tilde{\beta} - \varepsilon} \leq
  \frac{C}{\varepsilon} h^{\tilde{\beta} - \varepsilon}.
$$ 
Taking $\varepsilon = |\log h|^{-1}$ leads to the bound $C|\log h|
  h^{\tilde{\beta}}$. Then, using $\delta_i \eqsim h^\mu \eqsim H_i$
  for $x_i \in \mathcal{N}_h^0 \setminus \mathcal{N}_h^{0,\delta}$
  (see \eqref{eq:graded-H}), we have
\begin{equation} \label{eq:err-out}
  |\mathcal{L}_h[\mathcal{I}_hu](x_i) - \mathcal{L}[u](x_i)| \leq
  C\max\{h^{\mu s}, |\log h|^{\tilde{\kappa}} h^{\tilde{\beta}} \}
  \delta_i^{-2s} \quad \forall x_i \in \mathcal{N}_h^0 \setminus
  \mathcal{N}_h^{0,\delta}.
\end{equation}
  
\underline{Case 2: $x_i \in \mathcal{N}_h^{0,\delta}$.} We consider
  the three components of the $\delta$-interior consistency errors.
  (2-a) In view of Lemma \ref{lm:consistency-s1} ($\delta$-interior
  consistency of $\mathcal{L}_h^S$) and \eqref{eq:graded-H}, we have 
$$ 
\begin{aligned}
\left| 
\mathcal{L}_h^S[u](x_i) - \int_{\Omega_i} \frac{u(x_i) - u(y)}{|x_i -
  y|^{n+2s}} \,\mathrm{d}y \right| & \leq
C(\hat{\beta}) \delta_i^{s-\hat{\beta}} H_i^{\hat{\beta} - 2s} \\
& \eqsim C(\hat{\beta}) h^{(\hat{\beta} - 2s)\alpha} \delta_i^{s -
  \frac{(\hat{\beta} - 2s) \alpha}{\mu}} \delta_i^{-2s}.
\end{aligned}
$$ 
If $(\hat{\beta} - 2s) \alpha \leq \mu s$, then $ \delta_i^{s -
  \frac{(\hat{\beta} - 2s) \alpha}{\mu}} \lesssim 1$ due to $\delta_i
  \lesssim 1$.  Hence, we have the upper bound $C(\hat{\beta})
  h^{(\hat{\beta} - 2s)\alpha} \delta_i^{-2s}$, which turns out to be
  $C |\log h|^{\hat{\kappa}} h^{(\hat{\beta} - 2s)\alpha}
  \delta_i^{-2s}$.  Otherwise, if $(\hat{\beta} - 2s) \alpha > \mu s$,
  there exists $\hat{\beta}_0 \in [2s + \frac{\mu s}{\alpha},
  \hat{\beta}]$, such that neither $\hat{\beta}_0-2s$ nor $\hat{\beta}_0$ is
  an integer. Then, 
$$ 
\begin{aligned}
C(\hat{\beta}_0) h^{(\hat{\beta}_0 - 2s) \alpha} \delta_i^{s -
  \frac{(\hat{\beta}_0 - 2s) \alpha}{\mu}} \delta_i^{-2s}  &\leq
  C(\hat{\beta}_0) h^{\mu s} \left( h^\mu \delta_i^{-1} \right)^{
  \frac{(\hat{\beta}_0 - 2s) \alpha}{\mu} - s } \delta_i^{-2s} \\
& \leq C(\hat{\beta}_0) h^{\mu s} \delta_i^{-2s}.
\end{aligned}
$$ 
Here, $h^\mu \delta_i^{-1} \lesssim 1$ is used, since $h^\mu \lesssim
h_i \leq \delta_i$. Summing up two cases, we have 
\begin{equation} \label{eq:err-in1}
\left| 
\mathcal{L}_h^S[u](x_i) - \int_{\Omega_i} \frac{u(x_i) - u(y)}{|x_i -
  y|^{n+2s}} \,\mathrm{d}y \right| \leq C \max\{h^{\mu s}, |\log
  h|^{\hat{\kappa}} h^{(\hat{\beta} - 2s)\alpha} \}  \delta_i^{-2s}.
\end{equation}

(2-b) Similarly, in view of Lemma \ref{lm:consistency-Ih} (consistency of
   interpolation) and \eqref{eq:graded-H}, we have for any $x_i \in
   \mathcal{N}_h^{0,\delta}$,
   \begin{equation} \label{eq:err-in2}
   \begin{aligned}
     \left|\mathcal{L}_h^S[\mathcal{I}_hu](x_i) - \mathcal{L}_h^S[u](x_i)
  \right| &\leq C(\tilde{\beta})
  h^{\tilde{\beta}}\delta_i^{s-\frac{\tilde{\beta}}{\mu}}
  H_i^{-2s} \\
  & \eqsim C(\tilde{\beta}) h^{\tilde{\beta} - 2s \alpha} \delta_i^{s
     - \frac{\tilde{\beta} - 2s\alpha}{\mu}} \delta_i^{-2s} \\
  & \leq C \max\{ h^{\mu s}, |\log h|^{\tilde{\kappa}}
     h^{\tilde{\beta} - 2s\alpha} \} \delta_i^{-2s}.
  \end{aligned}
   \end{equation}
   
  (2-c)  For the $\delta$-interior consistency of $\mathcal{L}_h^T$
  (Lemma \ref{lm:consistency-t1}), the first term on the right hand
  side of \eqref{eq:consistency-t1} is the same as the case (2-b).
  For the other term, since $\tilde{\beta} \geq \mu s$, there exists
  $\tilde{\beta}_0 \in (\mu s, \mu(s+1))$ such that neither
  $\tilde{\beta}_0-2s$ nor $\tilde{\beta}_0$ is an integer. Then, 
  $$ 
  C(\tilde{\beta}_0) h^{\tilde{\beta}_0}
  \delta_i^{-\frac{\tilde{\beta}_0}{\mu} - s} = C(\tilde{\beta}_0)
  h^{\mu s} \left( h^\mu \delta_i^{-1}
  \right)^{\frac{\tilde{\beta}_0}{\mu} - s} \delta_i^{-2s} \leq C
  h^{\mu s} \delta_i^{-2s}.
  $$ 
  Otherwise $\tilde{\beta} \leq \mu s$, then $C(\tilde{\beta})
  h^{\tilde{\beta}} \delta_i^{-2s} \leq C |\log h|^{\tilde{\kappa}}
  h^{\tilde{\beta}} \delta_i^{-2s}$.
  As a result, 
  \begin{equation} \label{eq:err-in3}
   \left| \int_{\Omega_i^c} \frac{u(y) -
      \mathcal{I}_hu(y)}{|x_i -y |^{n+2s}} \,\mathrm{d}y \right| \leq
      C \max\{
      h^{\mu s},
       |\log h|^{\tilde{\kappa}} h^{\tilde{\beta} -2s\alpha} \}
       \delta_i^{-2s}.
  \end{equation}
Combining \eqref{eq:err-in1}-\eqref{eq:err-in3}, the consistency error
on $\delta$-interior nodes $x_i \in \mathcal{N}_h^{0,\delta}$ is given
as
\begin{equation} \label{eq:err-in}
  |\mathcal{L}_h[\mathcal{I}_hu](x_i) - \mathcal{L}[u](x_i)| \leq C
  \max\{h^{\mu s},  |\log h|^{\hat{\kappa}} h^{(\hat{\beta} - 2s)
  \alpha}, |\log h|^{\tilde{\kappa}} h^{\tilde{\beta} - 2s\alpha} \}
  \delta_i^{-2s}.
\end{equation} 
  
The numerical solution satisfies $\mathcal{L}_h[u_h](x_i) = f(x_i) =
\mathcal{L}[u](x_i)$, $\forall x_i \in \mathcal{N}_h^0$.  Invoking the
discrete barrier function $b_h$ defined in Lemma \ref{lm:barrier}, the
estimates \eqref{eq:err-out} and \eqref{eq:err-in}
yield 
$$ 
\begin{aligned}
& \quad~ |\mathcal{L}_h[\mathcal{I}_hu - u_h](x_i)| \\
& \leq C \max\{h^{\mu s},
|\log h|^{\hat{\kappa}} h^{(\hat{\beta} - 2s) \alpha}, |\log
h|^{\tilde{\kappa}} h^{\tilde{\beta} - 2s\alpha} \}
\mathcal{L}_h[b_h](x_i) 
\quad 
\forall x_i \in \mathcal{N}_h^0,
\end{aligned}
$$ 
whence, by Lemma \ref{lm:dcp} (discrete comparison principle)
$$ 
\max_{x_i \in \mathcal{N}_h^0}|\mathcal{I}_h u (x_i) - u_h(x_i)| \leq
C \max\{h^{\mu s},  |\log h|^{\hat{\kappa}} h^{(\hat{\beta} - 2s)
\alpha}, |\log h|^{\tilde{\kappa}} h^{\tilde{\beta} - 2s\alpha} \} .
$$ 
The desired estimate \eqref{eq:err-h1} then follows from the
approximation result \eqref{eq:global-approx}. 

The optimal $\alpha$ satisfies $(\hat{\beta} - 2s) \alpha_{\rm{opt}} =
\tilde{\beta} - 2s \alpha_{\rm{opt}}$, which gives $\alpha_{\rm{opt}}
= \tilde{\beta} / \hat{\beta}$. It is straightforward to see that
$\alpha_{\rm{opt}} \in [\frac12, 1]$ from the definitions of
$\tilde{\beta}$ and $\hat{\beta}$. Taking the optimal $\alpha$ into
\eqref{eq:err-h1} leads to \eqref{eq:err-h2}. The proof is therefore
complete.
\end{proof}

\begin{remark}[Huang-Oberman's work \cite{huang2014numerical}
  revisited] In \cite{huang2014numerical}, the 1D uniform grid with
  $H_i = h$ was applied, namely $\mu = 1$ and $\alpha = 1$. Then,
  \eqref{eq:err-h1} implies the pointwise error
  $\mathcal{O}(\max\{h^s, |\log h|^{\tilde{\kappa}} h^{\tilde{\beta} -
  2s} \})$, which is observed in Table \ref{tb:huang-oberman} since
  $\tilde{\beta} \leq 2$.
\end{remark}

Using the relationship between $h$ and the total number of degrees of
freedom \eqref{eq:graded-dof}, we obtain the following theorem.
\begin{theorem}[Convergence rates in terms of $N$]  \label{tm:err-N}
Let $\Omega$ be a bounded Lipschitz domain with exterior ball
condition. Let $f \in L^\infty(\Omega)$ be such that   
$\|f\|_{\beta-2s;\Omega}^{(s)} < \infty$. Then, on the graded grids \eqref{eq:graded-h} and $H_i$ chosen as in
 \eqref{eq:Hi}, $\alpha$ as $\alpha_{\rm{opt}} = \tilde{\beta} /
  \hat{\beta}$, and 
\begin{equation} \label{eq:mu-opt}
\mu \left\{
\begin{array}{ll}
\in \left[\frac{\tilde{\beta}}{s} - \frac{2\tilde{\beta}}{\hat{\beta}}, \frac{n}{n-1}\right) &
  \mbox{if } \left(\frac{n}{n-1} +
  \frac{2\tilde{\beta}}{\hat{\beta}}\right) s > \tilde{\beta}, \\
  = \frac{n}{n-1} & \mbox{if } \left(\frac{n}{n-1} +
  \frac{2\tilde{\beta}}{\hat{\beta}}\right) s \leq \tilde{\beta}.
\end{array}
\right.
\end{equation}
Then, the convergence rate 
\begin{equation} \label{eq:err-N}
\| u - u_h\|_{L^\infty(\Omega)} \leq C \left\{
\begin{array}{ll}
(\log N)^{\max\{\hat{\kappa}, \tilde{\kappa}\}}
  N^{-\frac{1}{n}(\tilde{\beta} -
  \frac{2\tilde{\beta}}{\hat{\beta}}s)} &  \mbox{if }
  \left(\frac{n}{n-1} + \frac{2\tilde{\beta}}{\hat{\beta}}\right) s >
  \tilde{\beta}, \\
(\log N)^{\max\{\hat{\kappa}, \tilde{\kappa}\} + \frac{s}{n-1}}
  N^{-\frac{s}{n-1}} & \mbox{if } \left(\frac{n}{n-1} +
  \frac{2\tilde{\beta}}{\hat{\beta}}\right) s = \tilde{\beta}, \\
(\log N)^{\frac{s}{n-1}} N^{-\frac{s}{n-1}} & \mbox{if }
  \left(\frac{n}{n-1} + \frac{2\tilde{\beta}}{\hat{\beta}}\right) s <
  \tilde{\beta},
\end{array}
\right.
\end{equation}
where the constant depends on $\Omega$, $s$, $\beta$,
$\|f\|_{L^\infty(\Omega)}$ and $\|f\|_{\beta -2s;\Omega}^{(s)}$.
\end{theorem}

\section{Numerical Experiments} \label{sc:numerical}
In this section, we present some numerical experiments in both one and
two-dimensional domains, which illustrate the sharpness of our
theoretical estimates.  In all of the experiments below, we set $\Omega
= B_1(0) \subset \mathbb{R}^n,\; n = 1,2$ and $f \equiv 1$, so that we
have an explicit solution
\[
\begin{aligned}
  u(x) = \frac{2^{-2s}\Gamma(n/2)}{\Gamma(n/2+s)\Gamma(1+s)} \left(1 -
  |x|^2 \right)_+^s\quad \forall x\in\Omega.
\end{aligned}
\]
This corresponds to smooth right hand side and the discussion of
Section \ref{sc:error} applies with \commentone{$\tilde{\beta}=2$, $\hat{\beta}=4$, $\tilde{\kappa}=\hat{\kappa} = 1$}, and $\alpha_{\rm{opt}} =
\frac{\tilde{\beta}}{\hat{\beta}} = \frac12$. 

\subsection{1D test}
When $n =1$, we always have the relation $N \eqsim h^{-1}$ for any
$\mu \geq 1$ due to \eqref{eq:graded-dof}. In the approximation of the
singular integral, the domain $\Omega_i$ in \eqref{eq:singular-domain}
is taken as the open interval centered at $x_i$ with radius $H_i$,
namely, $\Omega_i = (x_i - H_i, x_i + H_i)$. \commentone{On any element $T$ for which $T \cap \Omega_i^c \neq \varnothing$, the intersection is still an inverval so that the weight can be calculated explicitly, see \cite[Section 3]{huang2014numerical}.} In this case, the
convergence rate estimate \eqref{eq:err-h2} given by Theorem
\ref{tm:err-h} turns out to be 
\begin{equation} \label{eq:err-test1}
\|u - u_h\|_{L^\infty(\Omega)} \leq C \max\{ h^{\mu s}, |\log h| h^{2-s} \}.
\end{equation}

We start with the numerical tests on quasi-uniform grids ($\mu=1$).
The convergence rates for several values of $s$ are listed in Table
\ref{tab:uni1d}, and the  computational errors for $s = 0.3$, $s =0.6$
and $s = 0.9$ are shown in Figure \ref{fig:uni1d}. In all cases, we
see good agreement with the convergence rate $\mathcal{O}(h^s)$
predicted by \eqref{eq:err-test1}.

	\begin{figure}[!htbp]
		\centering
		\subfloat[Convergence rates]{
		\centering
		\scalebox{0.8}{
		\begin{tabular}[b]{|c|c||c|c|}
				\hline
				$s$&Rate& $s$ & Rate\\
				\hline
				0.1 & 0.10&0.6 & 0.60\\
				0.2 & 0.20&0.7 & 0.71\\
				0.3 & 0.30&0.8 & 0.81\\
				0.4 &  0.40&0.9 & 0.94\\ 
				0.5 & 0.50 & &\\
				\hline
		\end{tabular}
		}
		\label{tab:uni1d}	
		} %
		\subfloat[$\log(L^\infty\text{-error}) - \log(h)$]{
		\centering
		\includegraphics[width=0.5\textwidth]{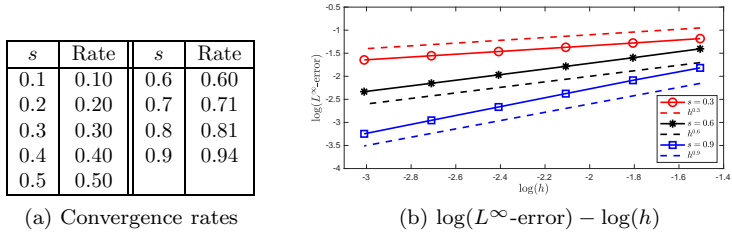}
		\label{fig:uni1d}
		}
	\caption{Convergence rates for problem (\ref{eq:fL}) using uniform grids in 1D case.}
	\end{figure}
	
We next consider numerical approximations using graded grids that
satisfying (\ref{eq:graded-h}) with $\mu =(2-s)/s$. We would expect a
convergence rate of order $\mathcal{O}(h^{2-s})$ (up to a logarithmic
factor) according to \eqref{eq:err-test1}. In Figure \ref{fig:grad1d}
we display the computational rates of convergence for $s = 0.4, 0.6,
0.8$, which are in good agreement with the theory.

\begin{figure}[!htbp]
	\begin{minipage}[b]{.45\linewidth}
		\centering
		\includegraphics[width = \textwidth]{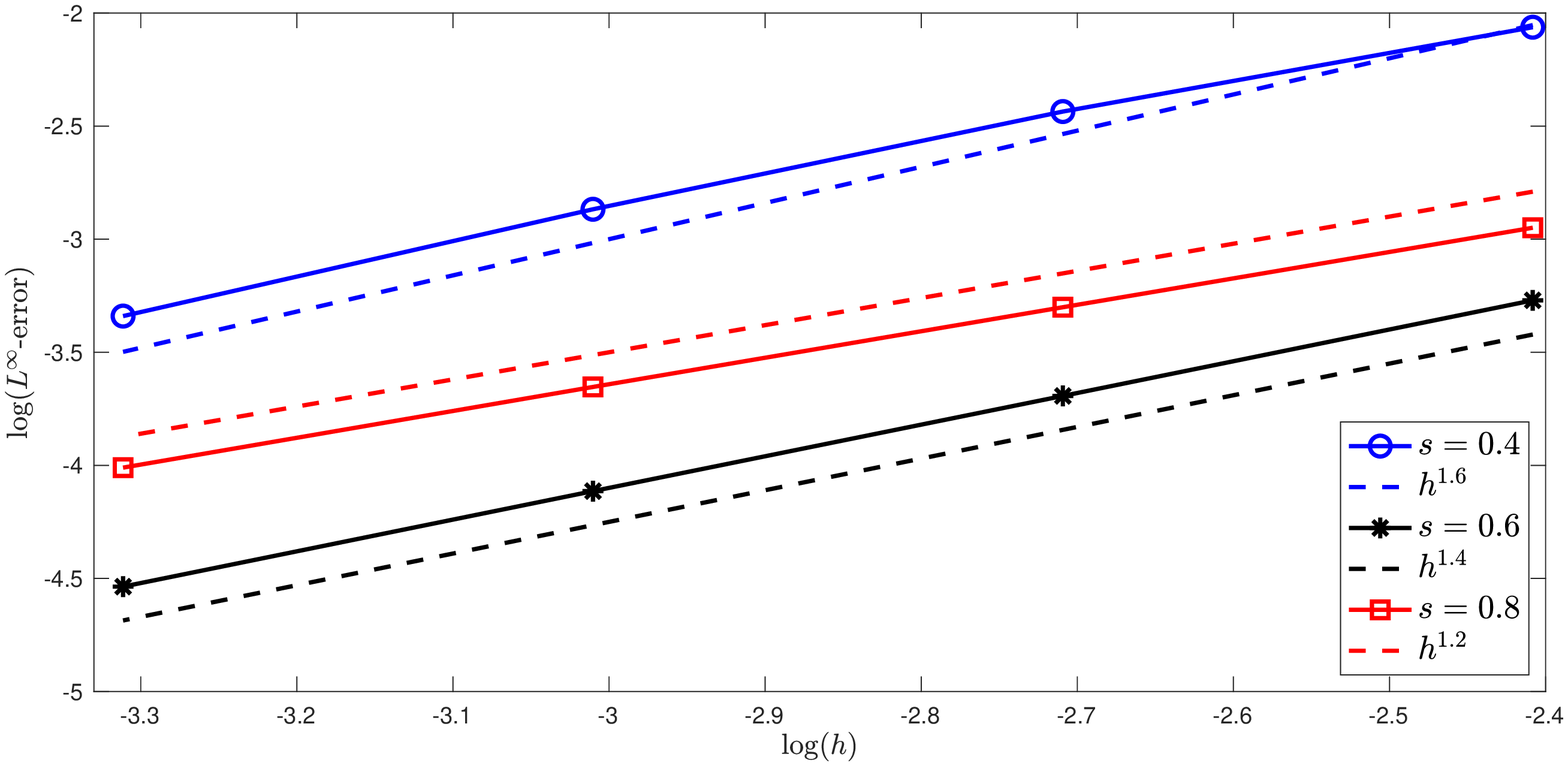}
    \caption{Convergence rates for graded grids with $\mu =
    \frac{2-s}{s}$ in 1D case.  Convergence rate of $2-s$ is
    observed.}
		\label{fig:grad1d}
	\end{minipage}
	\hspace{+0.2cm}
	\begin{minipage}[b]{.45\linewidth}
		\centering
		\includegraphics[width = \textwidth]{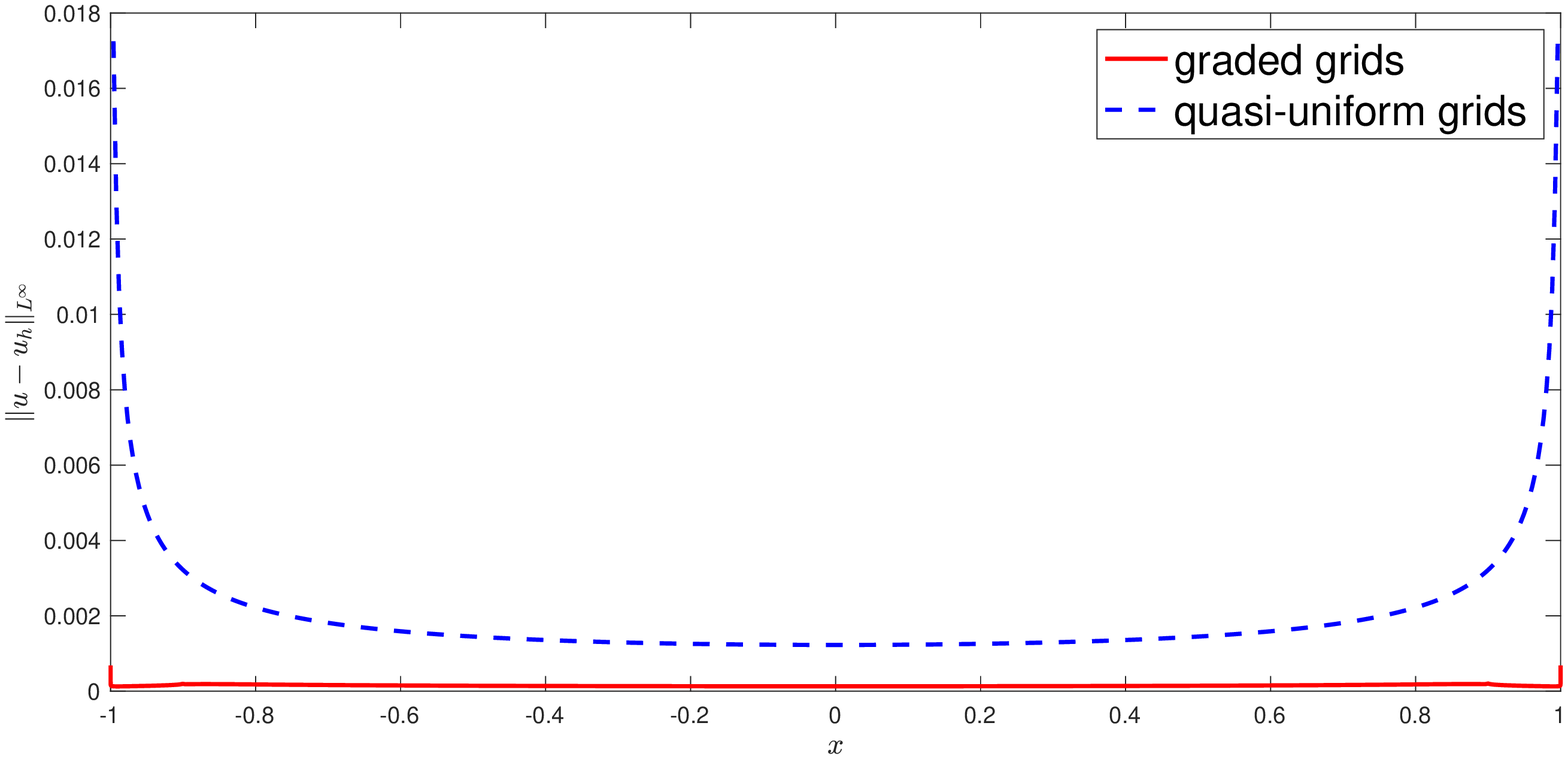}
    \caption{$\|u-u_h\|_{L^\infty}$ on quasi-uniform and graded grids.
    The maximal errors are embodied around the boundary.}
		\label{fig:err}
	\end{minipage}
\end{figure}

Next, we plot the $L^\infty$-errors for both uniform and graded grids
in Figure \ref{fig:err}. We observe that the $L^\infty$-error
increases rapidly near the boundary on quasi-uniform grids, due to the
poor H\"older continuity near the boundary. The graded grids will
alleviate this effect. Further, the error behaviors make it possible
to establish some improved interior (or local) pointwise error
estimates, which belong to our future work.

\subsection{2D test}
In the 2D test, $\Omega_i$ is taken as the square centered at $x_i$ with
side length $\sqrt{2} H_i$, see Remark \ref{rk:omega-i} (examples of
$\Omega_i$). An immediate benefit is the convenient numerical
integration on $\Omega_i^c$ on unstructured grids in \eqref{eq:Lh}.
\commentone{More precisely, the intersection of $T \cap \Omega_i^c$, if not empty, will become a polygon denoted by $P$. Therefore, the calculation of weight turns out to be the approximation of $C_{2,s}\int_P \phi(y) |y|^{-2-2s}\,\mathrm{d}y$, where $\phi$ is a linear function. Let $F(y) := \frac{C_{2,s}}{4s^2}|y|^{-2s}$ so that $\Delta F(y) = C_{2,s} |y|^{-2-2s}$. Then, after integration by parts twice, we obtain
\[
		C_{2,s} \int_{P} \phi(y) |y|^{-2-2s} \mathrm{d}y  = \int_{\partial P} \phi\,\frac{\partial F}{\partial n}\mathrm{d}s - \int_{\partial P} F\,\frac{\partial \phi}{\partial n}\mathrm{d}s.
		\]
The integral has been transformed into one-dimensional intervals, where the high-order numerical quadrature can be applied. 
}

\begin{figure}[!htbp]
\centering
	\subfloat[2D quasi-uniform grids]{
        \centering
	\includegraphics[width = 0.35\textwidth]{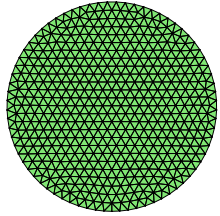}	
	\label{fig:unigrid}
	} %
        \subfloat[$\log(L^\infty\text{-error}) - \log(N^{-1})$]{
	\centering
	\includegraphics[width = 0.5\textwidth]{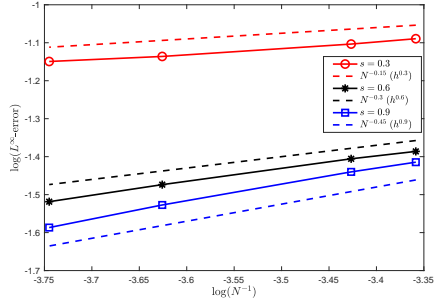}
	\label{fig:uni2d}
	}
	\caption{Convergence rates for problem (\ref{eq:fL}) using quasi-uniform grids in 2D case.} \label{fig:uni2d-result}
\end{figure}

\begin{figure}[!htbp]
\centering
	\subfloat[2D graded grids]{
        \centering
	\includegraphics[width = 0.35\textwidth]{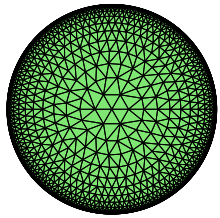}	
	\label{fig:gradgrid}
	} %
        \subfloat[$\log(L^\infty\text{-error}) - \log(N^{-1})$]{
	\centering
	\includegraphics[width = 0.5\textwidth]{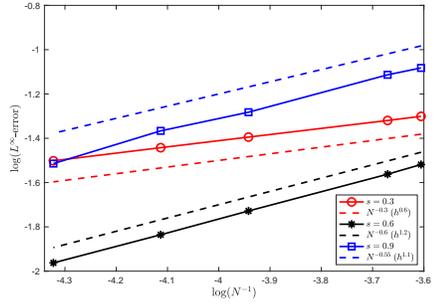}
	\label{fig:grad2d}
	}
	\caption{Convergence rates for problem (\ref{eq:fL}) using graded grids in 2D case.} \label{fig:grad2d-result}
\end{figure}

We next explore the sharpness of our estimates derived in Section
\ref{sc:error}. 
On a sequence of quasi-uniform grids (Figure \ref{fig:unigrid}), the
plots of $L^\infty$-errors for several values of $s$ are given in
Figure \ref{fig:uni2d}.  In light of (\ref{eq:graded-dof}), we have
the relationship $h \approx N^{-\frac12}$ for quasi-uniform grids
($\mu = 1$).  From \eqref{eq:err-h2} in Theorem \ref{tm:err-h}, the
theoretical  convergence rate $\mathcal{O}(h^s)$ or
$\mathcal{O}(N^{-\frac{s}{2}})$ coincides with the numerical tests. 

In the last test, we consider the computation on graded grids (Figure
\ref{fig:gradgrid}), where the errors are computed with respect to the
total number of degrees of freedom $N$. According to Theorem
\ref{tm:err-N}, the expected convergence rates are
$\mathcal{O}(N^{-\frac{2-s}{2}} )$ for $s > \frac{2}{3}$, and
$\mathcal{O}(N^{-s})$ for $s \leq \frac23$, up to a logarithmic
factor.  In Figure \ref{fig:grad2d} we display the computational rates
of convergence for $s = 0.3, 0.6, 0.9$, which are in good agreement
with theory.

\bibliographystyle{siamplain}
\bibliography{fracFD_SINUM}{} 
\end{document}